\documentclass[a4paper,11pt]{article}
\pdfoutput=1
\usepackage[usenames,dvipsnames]{color}
\usepackage{graphicx}
\usepackage{camnum}
\usepackage{harvard}
\usepackage{amssymb}
\usepackage{amsmath}

\newcommand{\ee}{{\mathrm e}}
\newcommand{\ii}{{\mathrm i}}
\newcommand{\ad}{\mathrm{ad}}
    
\newcommand{\dexp}{\mathrm{dexp}}

\def\OO#1{{\cal O}\left(#1\right)}
\newcommand{\dx}{\partial_x}

\newcommand{\ve}{\varepsilon}
\newcommand{\schr}{Schr\"{o}dinger }

\newcommand{\Ang}[2]{\left\langle #2 \right\rangle_{#1}}

\newcommand{\AAA}{\mathcal{A}}
\newcommand{\BBB}{\mathcal{B}}
\newcommand{\CCC}{\mathcal{C}}
\newcommand{\DDD}{\mathcal{D}}

\newcommand{\III}{\mathcal{I}}
\newcommand{\KKK}{\mathcal{K}}

\newcommand{\MMM}{\mathcal{M}}

\newcommand{\PPP}{\mathcal{P}}

\newcommand{\SSS}{\mathcal{S}}
\newcommand{\TTT}{\mathcal{T}}

\newcommand{\ang}[2]{\left\langle #1 \right\rangle_{#2}}

\newtheorem{defn}{Definition}
\numberwithin{equation}{section}
\newtheorem{cor}[theorem]{Corollary}

\begin{document}
\title{Algebraic theory for higher-order methods in computational quantum mechanics}

\author{Pranav Singh\footnote{Department of Applied Mathematics and Theoretical Physics, University of Cambridge, Wilberforce Rd, Cambridge CB3 0WA, UK. E-mail: {\tt singh.pranav@gmail.com}, Tel: {\tt +44 (0)1223 337892}, Fax: {\tt +44 (0)1223 765900}.}}
\maketitle

\vspace*{6pt}

\begin{abstract}
We present the algebraic foundations of the {\em symmetric Zassenhaus algorithm} and some of its variants. These algorithms have proven effective in
devising higher-order methods for solving the time-dependent \schr equation in the semiclassical regime.
We find that the favourable properties of these methods derive
directly from the structural properties of a $\mathbb{Z}_2$-graded Lie algebra. Commutators in this Lie algebra can be simplified explicitly, leading to
commutator-free methods.
Their other structural properties are crucial in proving
unitarity, stability, convergence, error bounds and quadratic costs of Zassenhaus based methods. These algebraic structures have also found applications in Magnus expansion
based methods for time-varying potentials where they allow significantly milder constraints for convergence and lead to highly effective schemes.
The algebraic foundations laid out in this work
pave the way for extending higher-order Zassenhaus and Magnus schemes to other equations of quantum mechanics.
\end{abstract}

\setcounter{section}{0}
\section{Introduction}
Recently devised symmetric Zassenhaus splittings \cite{bader14eaf}  and Magnus--Zassenhaus methods \cite{IKS,MKO} for computationally solving semiclassical time-dependent \schr equations (TDSEs)
have proven highly effective in achieving
arbitrarily high orders of accuracy with costs growing quadratically in the desired order.
This is in contrast to Yoshida based methods \cite{yoshida}, whose costs grow exponentially with order. These new schemes are commutator-free and allow much larger time steps than otherwise thought possible
for convergence of Magnus and Zassenhaus expansions.

Underlying the effectiveness of these new methods are some algebraic structures which form the crux of this work. In a break from tradition in numerical solutions of partial differential equations, the methods of \cite{bader14eaf} are devised by working directly in the \emph{free Lie algebra} of undiscretised operators $\dx^2$ and $V(x)$. The quadratic costs of symmetric Zassenhaus methods are achievable due to the property of \emph{height reduction} in these Lie algebras. This property  entirely fails to materialise when working with matrices corresponding to the discretisation of $\dx^2$ and $V(x)$. Hence the structure of these free Lie algebras is crucial to the effectiveness of Zassenhaus schemes.

The characterisation of these algebras forms a major motivation of this work. It is this characterisation that will allow us to formally prove the quadratic growth in costs. We will also find that these algebras possess a $\mathbb{Z}_2$-graded structure that, along with the symmetry of its constituents, proves crucial for the numerical stability as well as the unitary evolution of the wave-packet.
Instead of characterising the free Lie algebra,
\[ \mathrm{FLA}(\{V(x), \dx^2\}), \]
however, we will attempt the characterisation of a bigger free Lie algebra,
\[ \mathrm{FLA}(\SSS \cup \PPP(\dx)), \]
generated by all functions in the function space $\SSS$ and all polynomials of the differential operator $\dx$.

These algebraic characterisations will be introduced in a more abstract form---in the context of an associative algebra $\AAA$ with a commutative subalgebra $\CCC$ and its Lie idealiser $\III$. The algebraic structures of the Zassenhaus
schemes for the TDSE are seen to be special cases of the Lie algebras introduced here.
To be more concrete, the associative algebra $\AAA$ in the case of the TDSE will consist of linear operators, the commutative subalgebra will be formed by multiplicative operators (whose action is multiplication by $V$, for instance), and $\dx$ will be seen to be an element of the Lie idealiser  $\III$. The properties of these Zassenhaus schemes are traced to the structural properties of these algebras.

This algebraic formulation will prove of crucial importance in cutting the Zassenhaus algorithms free of their semiclassical TDSE origins, allowing their
use in devising higher-order computational methods for other equations of quantum mechanics and possibly beyond.

In Section~\ref{sec:defn} we introduce the abstract context in which our algebra $\mathfrak{F}$ is defined. In Section~\ref{sec:ass}, we show that $\mathfrak{F}$ is
an associative algebra and study its structure. While it follows immediately that $\mathfrak{F}$ is also a Lie algebra, in Section~\ref{sec:lie} we find out
that it possesses a very interesting structure which has significant ramifications for numerical methods discussed in Section~\ref{sec:app}.

The tables of coefficients in Appendix~\ref{app:tables} should aid direct applications of the results of this paper. The lengthier proofs of Lemma~\ref{lem:explicit_pi} and Lemma~\ref{lem:genfun_pi} from Section~\ref{sec:ass} have been confined to Appendices \ref{app:explicit_pi} and \ref{app:genfun_pi}, respectively, for ease of narration.

\section{Some notations and definitions}
\label{sec:defn}
Consider a commutative algebra $\CCC$ which is a subalgebra of the unital associative algebra $(\AAA,\cdot,+)$ over the field of rational numbers $\mathbb{Q}$. The commutator on the associative product,
    \[ [a,b] = a \cdot b - b \cdot a, \]
acts as the canonical Lie product while the anticommutator,
    \[ a \bullet b = \frac12 \left(a \cdot b + b \cdot a \right),\]
acts as a Jordan product.
$\AAA$ along with the Lie (Jordan) product forms a Lie (Jordan) algebra which we can identify with $\AAA$ again. The Lie idealiser of $\CCC$ in $\AAA$,
    \[ \III = \{ d \in \AAA\ :\ [d,\CCC] \subseteq \CCC \}. \]
is the largest subalgebra of $\AAA$ in which $\CCC$ is a Lie ideal.

In this work we typically use $a,b,c$ for elements of the associative algebra $\AAA$, $x,y,z$ for elements of the commutative subalgebra $\CCC$ and reserve $d$ for elements of the Lie idealiser $\III$, except where it is clear from context. For elements $x,y$ of the commutative algebra $\CCC$, we write $xy$ to denote $x \cdot y$, dropping the explicit use of the multiplication operator.
The letters $k,l,m,n,p,r,s,i$ are used for non-negative integers.

Consider the adjoint map, where the action of $ad_a$ for $a \in \AAA$ is described by
    \[ \ad_a(b) = [ a , b ] .\]
$\ad_a$ is always a derivation on $\AAA$ and is called an inner derivation. However, it need not be a derivation on $\CCC$ in general. For $d\in \III$, $\ad_d$ is also a derivation on $\CCC$,
    \[\ad_d \in \mathrm{Der}(\CCC), \qquad d \in \III.\]

\begin{defn}
    \label{def:ang}
    For $x \in \CCC, d\in \III$ and a non-negative integer $k$ we introduce the notation
    \[ \ang{x}{k}^d = x \bullet d^k = \frac12 (x \cdot d^k + d^k \cdot x), \]
    where the power $d^k$ is defined as usual, $d^{k+1}=d \cdot d^k$ and $d^0 = 1_{\AAA}$.
\end{defn}

\begin{defn}
    For any $d\in \III$ and non-negative integer $k$, we define the linear space
    \[\mathfrak{F}^d_k = \left\{  \ang{x}{k}^d \ :\ x \in \mathcal{C} \right\},\]
    and denote the direct sum of these linear spaces by
    \[\mathfrak{F}^d = \bigoplus_{k \in \mathbb{Z}_0^+} \mathfrak{F}^d_k. \]
\end{defn}

In this work we will only use a single non-trivial element of $\III$ which is not in $\CCC$, $d \in \III\setminus\CCC$, and can drop the superscript $d$, writing $\ang{x}{k}$, $\mathfrak{F}_k$ and $\mathfrak{F}$
in place of $\ang{x}{k}^d$, $\mathfrak{F}^d_k$ and $\mathfrak{F}^d$, respectively. To simplify notation further, we denote $\ad_d^i(x)$ by $D^i x$.

Since $D$ is a derivation on $\AAA$ (and $\CCC$), $D(a \cdot b) = D a \cdot b + a \cdot D b, $
and it distributes binomially on $\AAA$ (and $\CCC$),
\begin{equation}
     \label{eq:binomD} D^k (a \cdot b) = \sum_{i=0}^k \binom{k}{i} D^i a\cdot D^{k-i} b.
\end{equation}
Starting from $d \cdot x = D x + x \cdot d$, a simple inductive procedure leads to a similar binomial identity,
\begin{equation}
    \label{eq:binom}
    d^k \cdot x = \sum_{i=0}^k \binom{k}{i} D^i x \cdot d^{k-i}.
\end{equation}

\section{$\mathfrak{F}$ as an associative algebra}
\label{sec:ass}
\begin{lemma}
    \label{lem:angass}
    The linear space $\mathfrak{F}$ is an associative algebra with
    \begin{equation}
    \label{eq:angass}
        \ang{x}{k} \cdot \ang{y}{l} = \sum_{n=0}^{k+l} \ang{z_n}{k+l-n},
    \end{equation}
    where the terms
    \begin{equation}
        \label{eq:zs}
        z_n = \sum_{i=0}^{n} \pi_{n,i}^{k,l}\,D^i x\, D^{n-i} y
    \end{equation}
    are in $\CCC$ and $\pi_{n,i}^{k,l} \in \mathbb{Q}$.
\end{lemma}
\begin{proof}
    \normalfont
    A proof for Lemma~\ref{lem:angass} follows by expanding \R{eq:angass} using the binomial identity \R{eq:binom} and comparing powers of $d$. At first, we expand the left hand side,
    \begin{Eqnarray*}
        4 \ang{x}{k} \cdot \ang{y}{l}  & = & (x \cdot d^k + d^k \cdot x) \cdot 2 \ang{y}{l} \nonumber \\
        & = & x \cdot (d^k \cdot y) \cdot d^l + x \cdot (d^{k+l} \cdot y) + (d^k \cdot x) \cdot 2 \ang{y}{l}\\
        & = & x \cdot \left(\sum_{i=0}^k \binom{k}{i} D^i y \cdot d^{k-i} \right)\cdot d^l + x \cdot \sum_{i=0}^{k+l} \binom{k+l}{i} D^i y \cdot d^{k+l-i}\\
        &&+ \left(\sum_{i=0}^k \binom{k}{i} D^i x \cdot d^{k-i}\right) \cdot \left( y \cdot d^l + d^l \cdot y \right)\\
        &=& \sum_{i=0}^{k+l} \left[ \binom{k}{i} + \binom{k+l}{i}\right] (x\ D^i y) \cdot d^{k+l-i}\\
        && + \sum_{i=0}^{k} \sum_{j=0}^{k+l-i} \binom{k}{i}  \left[ \binom{k-i}{j} + \binom{k+l-i}{j} \right] (D^i x\  D^j y) \cdot d^{k+l-i-j}.
    \end{Eqnarray*}
    The right hand side of (\ref{eq:angass}) is expanded in similar fashion,
    \begin{equation*}
    \label{eq:Rab}
    4 \sum_{n=0}^{k+l} \ang{z_n}{k+l-n} =  2 \sum_{n=0}^{k+l} z_n \cdot d^{k+l-n} + 2 \sum_{n=0}^{k+l} \sum_{i=0}^{k+l-n} \binom{k+l-n}{i} D^i z_n \cdot d^{k+l-n-i}.
    \end{equation*}
    A sequence of terms $z_n$ for which all terms accompanying $d^{k+l-a}$, $a \in \{0, \ldots, k+l\}$, match on both sides will certainly satisfy \R{eq:angass} although it need not be the unique solution. This leads to the relation
    \begin{Eqnarray}
        \label{eq:zs_match}
        2 z_a + 2 \sum_{n=0}^{k+l} \binom{k+l-n}{a-n} D^i z_n  & = &\left[ \binom{k}{a} + \binom{k+l}{a}\right] (x\ D^a y)  \\
        \nonumber &&+ \sum_{i=0}^{a} \binom{k}{a}  \left[ \binom{k-i}{a-i} + \binom{k+l-i}{a-i} \right] (D^i x\  D^{a-i} y),
    \end{Eqnarray}
    which can easily be fashioned into a recursive procedure for solving $z_a$, starting from $z_0 = x y$ (found by substituting $a=0$).

    The second part of the lemma states that $z_n$ of the form \R{eq:zs} satisfy \R{eq:angass} for some $\pi_{n,i}^{k,l} \in \mathbb{Q}$. After substituting this form in \R{eq:zs_match}, the left side is comprised of
    \begin{equation}
        \label{eq:R1}
        2 \sum_{i=0}^{a} \pi_{a,i}^{k,l}\, D^i x\ D^{a-i} y
    \end{equation}
    and
    \begin{align}
        \nonumber    & 2 \sum_{n=0}^{k+l} \binom{k+l-n}{a-n} D^{a-n} \left( \sum_{i=0}^{n} \pi_{n,i}^{k,l}\, D^i x\ D^{n-i} y \right)\\
        \label{eq:R2}   &\mbox{}\qquad\qquad\qquad = 2 \sum_{n=0}^{k+l} \binom{k+l-n}{a-n} \sum_{i=0}^{n} \pi_{n,i}^{k,l} \sum_{j=0}^{a-n} \binom{a-n}{j} D^{i+j} x\ D^{a-i-j} y,
    \end{align}
    where the inner term $D^{a-n} \cdot \left( D^i x\ D^{n-i} y \right)$ has been expanded using \R{eq:binomD}.

    We now equate terms accompanying $D^p x\ D^{a-p} y$ in \R{eq:zs_match}, noting that this would give a solution which need not be unique.
    We arrive at equations of the form
    \begin{equation*}
            R_{a,p}^{k,l} = L_{a,p}^{k,l}, \qquad p \in \{0, \ldots, a\},\quad a \in \{0, \ldots, k+l \},
    \end{equation*}
    where
    \begin{eqnarray}
        \label{eq:RHSlambda}
        R_{a,p}^{k,l} &=& 2 \pi_{a,p}^{k,l} + 2 \sum_{n=0}^{a} \binom{k+l-n}{a-n}  \sum_{i=0}^{n} \pi_{n,i}^{k,l} \binom{a-n}{p-i},\\
        \label{eq:LHSlambda}
        L_{a,p}^{k,l} &=& \delta_{p,0} \left( \binom{k}{a} + \binom{k+l}{a}\right) \\
        \nonumber && +\binom{k}{p} \left( \binom{k-p}{a-p} + \binom{k+l-p}{a-p} \right),
    \end{eqnarray}
    and $\delta_{i,j}$ is the Kronecker delta function. The fact that a recursive procedure for finding $\pi$s can be designed,
    \begin{equation}
        \label{eq:Asol}
        \pi_{a,p}^{k,l} = \frac14 \left(L_{a,p}^{k,l} - 2\sum_{n=0}^{a-1} \binom{k+l-n}{a-n}  \sum_{i=0}^{n} \pi_{n,i}^{k,l} \binom{a-n}{p-i} \right),
    \end{equation}
    starting from $\pi_{0,0}^{k,l} = 1$, is the proof that such $\pi_{n,i}^{k,l} \in \mathbb{Q}$ exist, whereby some $z_n \in \CCC$ of the form \R{eq:zs} satisfy \R{eq:angass}.
    Hence $\mathfrak{F}$ is an associative algebra with the prescribed properties \R{eq:angass} and \R{eq:zs}.
\end{proof}

\subsection{Explicit form of the coefficients}
The recursive procedure \R{eq:Asol} suffices for the proof of Lemma~\ref{lem:angass} and for expanding products in the associative
algebra $\mathfrak{F}$ using a symbolic computation algorithm. Coefficients derived using this procedure for a few combinations of $k$ and $l$ are listed in Table~\ref{tab:pi} in Appendix~\ref{app:tables}.

Nevertheless, an explicit form for these coefficients remains highly desirable for a variety of reasons.
Firstly, this could give us an explicit form for some analytic functions on $\mathfrak{F}$ (or some subspace of $\mathfrak{F}$), such as the exponential map.
Secondly, an analysis of growth of these coefficients could allow us to deduce estimates concerning the convergence of some of these functions.
Lastly, the observations which lead us to the structural properties of the Lie algebra in Section~\ref{sec:lie}---the main result of this paper---are not immediately evident in the recursive form.

We state here some results concerning the form of the coefficients, confining the proofs to appendix~\ref{app:explicit_pi} and appendix~\ref{app:genfun_pi} for ease of narration.
\begin{lemma}
The explicit form of the coefficients is given by
    \label{lem:explicit_pi}
    \begin{equation}
    \label{eq:explicit_pi}
    \pi_{n,i}^{k,l} = \frac12 \sum_{s=0}^{k+l} \sum_{j=0}^{k+l} A_{(n,i),(s,j)}^{k+l}\ L_{s,j}^{k,l},
    \end{equation}
where
    \begin{equation}
    A_{(n,i),(s,j)}^{q} = \begin{cases} \delta_{n,s}\delta_{i,j} - \frac{P_{n-s+1}}{n-s+1} \binom{q-s}{n-s} \binom{n-s}{i-j} & \qquad n \geq s,\ i \geq j, \\
    0 & \qquad \mathrm{otherwise},
    \end{cases}
    \end{equation}
and $P_r$ are defined in terms of the Bernoulli numbers $B_r$,
    \begin{equation*}
    \label{eq:P_r}
        P_r = (-1)^r (2^r -1) B_r.
    \end{equation*}
\end{lemma}

Some properties of the coefficients $\pi_{n,i}^{k,l}$ become evident once we study their generating functions. In this context we make an exception in notation, using $u,w,x$ and $y$ as variables in which the formal series of the generating function is specified.
\begin{lemma}
    \label{lem:genfun_pi}
    The generating function,
    \begin{equation}
    \label{eq:genfun_pi}
    h(u,w,y,x) = \sum_{l=0}^{\infty} \frac{u^l}{l!} \sum_{k=0}^{\infty} \frac{w^k}{k!} \sum_{n=0}^{k+l} (k+l-n)! y^n \sum_{i=0}^{n} x^i \pi^{k,l}_{n,i}\ ,
    \end{equation}
    for the coefficients $\pi_{n,i}^{k,l}$ appearing in (\ref{eq:angass}) is
    \begin{equation}
    \label{eq:genfun_sol_pi}
    h(u,w,y,x) = \frac{\exp\left((wy-uxy)/2\right) }{1-(w+u)} \frac{\cosh(uy/2)\cosh(wxy/2)}{\cosh(y(u+w)(1+x)/2)}.
    \end{equation}
\end{lemma}

\begin{lemma}
    \label{lem:symmetry_pi} The coefficients possess the symmetry,
    \begin{equation*}
        \pi_{n,i}^{k,l} = (-1)^n \pi_{n,n-i}^{l,k}.
    \end{equation*}
\end{lemma}
\begin{proof}\normalfont
    We show that the generating function for $(-1)^n \pi_{n,n-i}^{l,k}$ coincides with the generating function of $\pi_{n,i}^{k,l}$,
    \begin{Eqnarray*}
    \label{eq:genfun_structure_pi}
    g(u,w,y,x) &=& \sum_{l=0}^{\infty} \frac{u^l}{l!} \sum_{k=0}^{\infty} \frac{w^k}{k!} \sum_{n=0}^{k+l} (k+l-n)! y^n \sum_{i=0}^{n} x^i  (-1)^n \pi_{n,n-i}^{l,k} \\
    &=& \sum_{l=0}^{\infty} \frac{w^l}{l!} \sum_{k=0}^{\infty} \frac{u^k}{k!} \sum_{n=0}^{k+l} (k+l-n)! (-y)^n \sum_{i=0}^{n} x^{n-i} \pi_{n,i}^{k,l}\\
    &=& h(w,u, -xy,1/x)\\
    &=& \frac{\exp\left((-uxy+wy)/2\right) }{1-(w+u)} \frac{\cosh(-wxy/2)\cosh(-uy/2)}{\cosh(-yx(u+w)(1+1/x)/2)}\\
    &=& \frac{\exp\left((wy-uxy)/2\right)}{1-(w+u)} \frac{\cosh(wxy/2)\cosh(uy/2)}{\cosh(yx(u+w)(1+1/x)/2)}\\
    &=& h(u,w,y,x).
    \end{Eqnarray*}
\end{proof}

\section{$\mathfrak{F}$ as a Lie algebra}
\label{sec:lie}
Since $\mathfrak{F}$ is an associative algebra, it is immediately obvious that it is also a Lie algebra with the commutator as a canonical Lie product. From \R{eq:angass}, we know that the commutator can be expanded to
\begin{equation}
    \label{eq:mulie}
    \left[ \ang{x}{k} , \ang{y}{l} \right] = \sum_{n=0}^{k+l} \sum_{i=0}^{n} \mu_{n,i}^{k,l} \ang{D^i x\, D^{n-i} y}{k+l-n},
\end{equation}
where $\mu_{n,i}^{k,l} = \pi_{n,i}^{k,l} - \pi_{n,n-i}^{l,k}$. However, the structure of $\mathfrak{F}$ turns out to be more interesting than this.

\begin{theorem}
    \label{thm:anglie}
    $\mathfrak{F}$ is a Lie algebra where commutators can be solved explicitly using the rule
    \label{eq:lie}
    \begin{equation}
        \label{eq:anglie}
        \left[ \ang{x}{k} , \ang{y}{l} \right] = \sum_{n=0}^{\frac{k+l-1}{2}} \sum_{i=0}^{2n+1} \lambda_{n,i}^{k,l} \ang{D^i x\, D^{2n+1-i} y}{k+l-2n-1},
    \end{equation}
    where $\lambda_{n,i}^{k,l} = 2 \pi_{2n+1,i}^{k,l}\in \mathbb{Q}$.
\end{theorem}
\begin{proof}
    \normalfont The even indexed coefficients $\mu_{2n,i}^{k,l}$ in \R{eq:mulie} vanish due to  Lemma~\ref{lem:symmetry_pi}, $\mu_{2n,i}^{k,l} = \pi_{2n,i}^{k,l} - (-1)^{2n} \pi_{2n,i}^{k,l} =0$, while $\mu_{2n+1,i}^{k,l} = \pi_{2n+1,i}^{k,l} - (-1)^{2n+1} \pi_{2n+1,i}^{k,l} = 2\pi_{2n+1,i}^{k,l}$. We conveniently rename $\mu_{2n+1,i}^{k,l}$ as $\lambda_{n,i}^{k,l}$.
\end{proof}


We now look at some interesting properties of this Lie algebra. Consider the linear spaces,
    \[\mathfrak{G}_n = \bigoplus_{k \leq n} \mathfrak{F}_k.\]
We note that, as a consequence of Theorem~\ref{thm:anglie}, we have a natural filteration,
    \[ \left[ \mathfrak{G}_k, \mathfrak{G}_l \right] \subseteq \mathfrak{G}_{k+l-1}. \]
This property of {\em height reduction} is evident in \R{eq:anglie} where the largest index $k+l-1$ occurs for $s=0$. We define the height of a term in $\mathfrak{G}_n$ as,
    \[ \mathrm{ht}\left(\sum_{k=0}^{n} \ang{x_k}{k} \right) = n, \]
if $x_n$ is non-zero. We define the height of $0$ as $-1$, this being the only term with negative height.
Thus $\mathrm{ht}\left(\ang{x}{k}\right) = k$ and, as a consequence of \R{eq:anglie},
    \begin{equation}
        \label{eq:ht}
        \mathrm{ht}\left( \left[ \ang{x}{k}, \ang{y}{l} \right] \right) \leq k + l - 1.
    \end{equation}

\begin{cor}
\label{cor:ht}
    {\bf (Height Reduction)}
    For any commutator $C$ featuring the \emph{letters} $a_i \in \mathfrak{G}_{k_i},\ i=1,\ldots, n$,
        \begin{equation}
            \label{eq:grade}
            \mathrm{ht}\left(C\left(a_i, \ldots, a_n\right) \right) \leq \sum_{i=1}^n k_i - n + 1.
        \end{equation}
\end{cor}
Another useful property of this Lie algebra is that it doesn't {\em mix} terms of certain forms. We define
    \[\mathfrak{e} = \bigoplus_{k \geq 0} \mathfrak{F}_{2k}\ , \qquad \mathfrak{o} = \bigoplus_{k \geq 0} \mathfrak{F}_{2k+1}, \]
so that $\mathfrak{F} = \mathfrak{o} \bigoplus \mathfrak{e}$. The following relations are evident from (\ref{eq:anglie}),
    \begin{Eqnarray}
    \label{eq:superalgebra}
    \left[ \mathfrak{e}, \mathfrak{e} \right] \subseteq \mathfrak{o}, && \left[ \mathfrak{o}, \mathfrak{o} \right] \subseteq \mathfrak{o},\\
    \nonumber \left[ \mathfrak{e}, \mathfrak{o} \right] \subseteq \mathfrak{e}, && \left[ \mathfrak{o}, \mathfrak{e} \right] \subseteq \mathfrak{e}.
    \end{Eqnarray}
Thus, if each letter $a_i$ is either in $\mathfrak{e}$ or in $\mathfrak{o}$, the commutator $C\left(a_i, \ldots, a_n\right)$ is either in $\mathfrak{e}$ or $\mathfrak{o}$
and does not mix terms from the two. Moreover, commutators with the same letters  will fall in the same space, $\mathfrak{e}$ or $\mathfrak{o}$.
This has implications for certain numerical methods where this property translates into a matrix being either symmetric or skew-symmetric but not a mix of the two which would have resulted in unfavourable structures.

The property \R{eq:superalgebra}, in fact, says that $\mathfrak{F}$ is a $\mathbb{Z}_2$-graded Lie algebra with the degree $0$ component $\mathfrak{o}$ and degree $1$ component $\mathfrak e$.
The linear space $\mathfrak{o}$ is a Lie algebra in its own right, while structures of the form $\mathfrak{e}$ are also called Lie triple systems which are closed under double commutation:
    \[ \left[\mathfrak{e},\left[\mathfrak{e},\mathfrak{e}\right]\right] \subseteq \mathfrak{e}. \]
These notions are closely related to Lie groups and symmetric spaces which have found
applications in linear algebra methods such as the generalised polar decomposition \cite{gpd}.

As anticipated, there is a close relation between $\mathfrak{F}$ and the {\em free Lie algebra} generated by $\CCC$ and polynomials in $d$.
\begin{lemma}\label{lem:fla}
    The free Lie algebra generated by $\CCC$ and $\PPP(d)$ (the ring of polynomials in $d$ with constant coefficients) is contained in $\mathfrak{F}$,
    \[ \mathfrak{g} := \mathrm{FLA}(\CCC \cup \PPP(d)) \subseteq \mathfrak{F}. \]
    The two are identical if an inverse of the mapping $D = \ad_d: \CCC \rightarrow \CCC$ exists.
\end{lemma}
\begin{proof}
    The containment is not difficult to prove. Since $d^k = \ang{1}{k}$, it is contained in $\mathfrak{F}_k$ and therefore $\PPP(d) \subseteq \mathfrak{F}$. The algebra $\CCC$ is also contained in $\mathfrak{F}$ since every $x \in \CCC$ can be written in the form $x = \ang{x}{0}$. The free Lie algebra $\mathfrak{g}$ generated by $\CCC \cup \PPP(d)$ is the intersection of all Lie algebras containing $\CCC$ and $\PPP(d)$ and is therefore contained in the Lie algebra $\mathfrak{F}$.

    The two algebras $\mathfrak{g}$ and $\mathfrak{F}$ are identical if inverse of the map $D$ exists. Any term in $\mathfrak{F}_0$ is of the form $\ang{y}{0}$ for some $y \in \CCC$ and therefore trivially resides in $\mathfrak{g}$. Take any $x \in \CCC$ and note that $d^2 \in \PPP(d)$. These are both contained in $\mathfrak{g}$ and by definition of being a Lie algebra $[d^2, x] = [\ang{1}{2},\ang{x}{0}] = \lambda_{0,0}^{2,0} \ang{D^0(1) D^1(x)}{1} + \lambda_{0,1}^{2,0} \ang{D^1(1) D^0(x)}{1} = 2 \ang{D(x)}{1}$ resides in $\mathfrak{g}$ as well. Consequently, any $\ang{y}{1} \in \mathfrak{F}_1$ can be expressed as $\frac12 [ d^2, D^{-1}(y)]$ so long as the inverse of $D$ exists. Thus we have $\mathfrak{F}_0,\mathfrak{F}_1 \subseteq \mathfrak{g}$. These two cases form the base case of our induction argument.

    Assume that for all $k<n$, $\mathfrak{F}_{2k}$ and $\mathfrak{F}_{2k+1}$ are contained in $\mathfrak{g}$. For completing the inductive proof, consider
    \begin{Eqnarray*}
        [d^{2n+1},x] & = & [\ang{1}{2n+1}, \ang{x}{0}] = \sum_{s=0}^n \sum_{i=0}^{2s+1} \lambda_{s,i}^{2n+1,0} \ang{D^{i}(1) D^{2s+1-i}(x)}{2n-2s} \\
        &=& \sum_{s=0}^n \lambda_{s,0}^{2n+1,0} \ang{D^{2s+1} (x)}{2n-2s} \\
        &=& \lambda_{0,0}^{2n+1,0} \ang{D(x)}{2n} + \sum_{s=1}^n \lambda_{s,0}^{2n+1,0} \ang{D^{2s+1}(x)}{2n-2s}.
    \end{Eqnarray*}
    Setting $x=D^{-1}(y) \in \CCC$,
    \[ \ang{y}{2n} = \Frac{1}{\lambda_{0,0}^{2n+1,0}} [d^{2n+1},D^{-1}(y)] - \sum_{s=1}^n \Frac{\lambda_{s,0}^{2n+1,0}}{\lambda_{0,0}^{2n+1,0}} \ang{D^{2s}(y)}{2n-2s}. \]
    The first term on the right hand side is in $\mathfrak{g}$ since $d^{2n+1} \in \PPP(d)$, while the terms in the summation fall in $\mathfrak{g}$ due to the induction hypothesis. Thus, we find that $\ang{y}{2n} \in \mathfrak{g}$ for any $y \in \CCC$. A similar proof shows that $\ang{y}{2n+1}$ also resides in $\mathfrak{g}$, proving that $\mathfrak{F}_k \subseteq \mathfrak{g}$ for every $k$. Since $\mathfrak{g}$ is also a linear space, the direct sum of these spaces, $\mathfrak{F} = \bigoplus_{k \in \mathbb{Z}^+_0} \mathfrak{F}_k$, is also contained in it. This completes our proof, $\mathfrak{g}=\mathfrak{F}$.
\end{proof}

We note that, in general, the free Lie algebra $\mathfrak{g}$ is not a subalgebra of the Lie idealiser $\III$ since $d^2$ need not be in $\III$ for every $d \in \III$.

Of less immediate and practical interest to us is the fact that $\mathfrak{F}$ is also a Jordan algebra. This follows directly and trivially from the fact that it is an associate algebra (Lemma~\ref{lem:angass}). More interestingly, it is also a $\mathbb{Z}_2$-graded Jordan algebra since, along similar lines to \R{eq:mulie},
\begin{equation}
    \label{eq:mujordan}
    \ang{x}{k} \bullet \ang{y}{l} = \sum_{n=0}^{k+l} \sum_{i=0}^{n} \gamma_{n,i}^{k,l} \ang{D^i x\, D^{n-i} y}{k+l-n},
\end{equation}
where $\gamma_{n,i}^{k,l} = (\pi_{n,i}^{k,l} + \pi_{n,n-i}^{l,k})/2$ vanishes for odd values of $n$ due to  Lemma~\ref{lem:symmetry_pi}, while $\gamma_{2n,i}^{k,l} = \pi_{2n,i}^{k,l}$ survives. Corresponding observations about the \emph{free Jordan algebra} generated by $\CCC$ and $\PPP(d)$ can also be made along similar lines, however a property corresponding to height reduction of Corollary~\ref{cor:ht} does not follow.

\section{Some applications}
\label{sec:app}
The applications of the algebras of Section~\ref{sec:lie} in solving partial differential equations such as the time-dependent \schr equation will arise from treating function spaces as commutative algebras and linear differential operators as elements of the associative algebra of endomorphisms. With this motivation, we will restrict our attention to cases where the commutative algebra $\CCC$ is isomorphic to a function space.

Let $\SSS$ be a commutative algebra isomorphic to $\CCC$ and $\Theta$ be the isomorphism between them. Anticipating the case where $\SSS$ is a function space, we will use $f,g,h$ to denote its elements. For $f \in \SSS$, we use the notation $\Theta_f$ and $\Theta(f)$ interchangeably
for the corresponding element in $\CCC$.

For an element $d\in \III$  in the Lie idealiser of $\CCC$, $D=\ad_d$ is a derivation on $\CCC$.
For any $f\in \SSS$, $D(\Theta_f) = [d, \Theta_f] \in \CCC$ so that $\tilde{d}(f) := \Theta^{-1} (D(\Theta_f))$ is an element in $\SSS$.
Since $D$ is a derivation on $\CCC$ and $\SSS$ is isomorphic to $\CCC$, $\tilde{d}$ is a derivation on $\SSS$. We say that $\tilde{d}$ is a derivation induced by $d$. The element of the Lie idealiser that we will require in the following sections will be the differential operator $d=\dx$, which will end up coinciding with the induced derivation $\tilde{d}$.

Our first example of $\Theta$ will be $\MMM$, the left multiplication map, which we encounter while solving the TDSE. In the case of the Wigner equation, the isomorphism $\Theta$ is more complicated and maps functions to pseudo-differential operators.

\subsection{Semiclassical time-dependent \schr equation (TDSE)}
The linear time-dependent \schr equation (TDSE),
\begin{equation*}
    \ii \hbar \partial_t u(x,t) = - \frac{\hbar^2}{2m} \dx^2 u(x,t) + V(x) u(x,t), \quad x\in \mathbb{R},\ t\geq 0,\ u(x,0) = u_0(x),
\end{equation*}
is a fundamental equation of quantum mechanics. The reduced Planck's constant, $\hbar \approx 1.054 \times 10^{-34}\, \mathrm{J}\cdot \mathrm{s}$, is a truly minute number which would typically result in considerable difficulties as far as numerical solutions are concerned. However, when working in atomic units, we can re-scale $\hbar =1$. In these units the mass of an electron is $1$. The typical length scales in these units are $10^{-11}\, \mathrm{m}$, while typical time and mass scales are $10^{-17} \mathrm{s}$ and $10^{-30}\, \mathrm{kg}$, respectively. Consequently, working in the atomic units restricts us to extremely small spatio-temporal windows.

When computation for any larger spatio-temporal windows is required, we arrive at the TDSE under the semi-classical scaling,
\begin{equation}
    \label{eq:TDSE}
    \ii \ve \partial_t u(x,t) = - \ve^2 \dx^2 u(x,t) + V(x) u(x,t),\quad x\in \mathbb{R},\ t\geq 0,\ u(x,0) = u_0(x).
\end{equation}
The \emph{semiclassical} parameter $\ve$ plays a role similar to the reduced Planck's constant. Although it is very small, $0 < \ve \ll 1$, the semiclassical parameter is considerably larger than the Planck's constant and the range $10^{-8} \leq \ve \leq 10^{-2}$ isn't unrealistic.

The equation \R{eq:TDSE} also arises out of the {\em Born--Oppenheimer approximation} of the molecular \schr equation.
In this case it describes the evolution of the nuclei in the electric potential, $V(x)$. Working in the atomic units, $\ve$ is the square root of the ratio of the mass of an electron and the heaviest nucleus. When larger spatio-temporal windows are required, $\ve$ can decrease further in size as before.

Typical computational methods for solving this equation commence with imposition of periodic boundaries at $\pm1$ followed by a discretisation of space.
The infinite dimensional operators $\dx$ and $V$ are replaced by matrices $\KKK$ and $\DDD_V$, respectively. The TDSE is now
replaced by a system of ODEs,
    \[ \partial_t \MM{u}(t) = (\ii \ve \KKK^2 - \ii \ve^{-1} \DDD_V)\ \MM{u}(t),\quad t\geq 0,\ \MM{u}(0) = \MM{u_0}, \]
where $\MM{u}$ is the vector of values at grid points. The exact solution of this equation can be written as
    \[ \MM{u}(t) = \exp(\ii t \ve \KKK^2 - \ii t \ve^{-1} \DDD_V)\ \MM{u_0}. \]
Methods for evaluating the exact matrix exponential for this matrix are prohibitively expensive and one resorts to splitting methods such as the Trotter splitting,
    \[ \exp(\ii t \ve \KKK^2 - \ii t \ve^{-1} \DDD_V) = \exp(\ii t \ve \KKK^2) \exp(-\ii t \ve^{-1} \DDD_V) +\OO{t^2}, \]
or the Strang splitting,
\begin{equation}
    \label{eq:strang}
    \exp(\ii t \ve \KKK^2 - \ii t \ve^{-1} \DDD_V) = \exp\left(\Frac12  \ii t \ve \KKK^2\right) \exp(-\ii t \ve^{-1} \DDD_V) \exp\left(\Frac12 \ii t \ve \KKK^2\right) +\OO{t^3}.
\end{equation}
Methods with higher-order accuracy than the Strang splitting can be created using Yoshida composition \cite{yoshida}. However the number of exponentials
grows exponentially with the order desired. Other methods are based on the \emph{symmetric Baker-Campbell-Hausdorff} (sBCH) formula \cite{dynkin47eot,casas09aea},
\begin{equation*}
    \exp\left(\Frac12 t A\right) \exp(t B) \exp\left(\Frac12 t A\right) = \exp(\mathrm{sBCH}(tA, tB)),
\end{equation*}
where
\begin{equation}
    \label{eq:sBCHcoeffs}
    \CC{sBCH}(t A,t B) = t (A+B) - t^3(\Frac{1}{24}[[B,A],A]+\Frac{1}{12}[[B,A],B]) +\OO{t^5}.
\end{equation}
These require the evaluation of nested commutators of the matrices $\KKK^2$ and $\DDD_V$ such as $[[\DDD_V, \KKK^2], \KKK^2]$, which is prohibitive and therefore typically avoided.

In \cite{bader14eaf} we develop arbitrarily high order methods by working directly within the free Lie algebra generated by the undiscretised operators $\dx^2$ and $V$.
The cost of these methods grows quadratically in contrast to Yoshida methods. Additionally, these methods are commutator-free and preserve unitary evolution
of the wave function---a fundamental principle of quantum mechanics. In this section we discover how the remarkable properties of symmetric Zassenhaus methods of \cite{bader14eaf} arise from the structural properties of (a special case of) the Lie algebra $\mathfrak{F}$.

\subsubsection{Notation}
Let $\SSS \subseteq \TTT$ be real and complex valued function spaces, respectively. Endomorphisms on $\TTT$ form an associative algebra $\AAA=(\mathrm{End}(\TTT),\circ,+)$ where $\circ$ is operator composition. For convenience, we consider $\SSS = \CC{C}_p^\infty([-1,1]; \mathbb{R})$ and
$\TTT = \CC{C}_p^\infty([-1,1]; \mathbb{C})$,
the space of smooth periodic functions over $[-1,1]$ with values in $\mathbb{R}$ and $\mathbb{C}$, respectively.
Let $M: \SSS \rightarrow \mathrm{End}(\TTT)$ be the left multiplication map,
\[ M(f)(g) = f g \in \TTT, \qquad f \in \SSS,\ g \in \TTT. \]
The image of $\SSS$ under $M$,
\[ \CCC = M(\SSS), \]
is a commutative algebra of multiplication operators with $M$ acting as an isomorphism between $\SSS$ and $\CCC$.
We also write $M_f$ to denote the operator $M(f)$ whose action is that of multiplying by $f$.
The partial differentiation operator $\partial_x$
is an element of the Lie idealiser of $\CCC$ in the associative operator algebra $(\mathrm{End}(\TTT), \circ, +)$,
\[ [\partial_x, M_f] = \partial_x \circ M_f - M_f \circ \partial_x = M_{\partial_x f} + M_f \circ \partial_x - M_f \circ \partial_x = M_{\partial_x f} \in \CCC, \]
where the commutator is the canonical Lie product as usual. Following our terminology, the operator $\dx \in \mathrm{End}(\TTT)$
induces the derivation $\tilde{\dx} = \dx$ on $\SSS$. Here the induced derivation $\tilde{d}$ overlaps with the element $d$, but this need not be the case in general. To be very precise, $\tilde{\dx}$ is an operator on $\SSS$ while $\dx$ is an operator on $\TTT$, but the two coincide on $\SSS$.


Lemma~\ref{lem:fla} directly allows us to conclude that
\[ \mathrm{FLA}(\MMM(\SSS) \cup \mathcal{P}(\partial_x) )\ = \ \mathfrak{F}^{\dx} = \bigoplus_{k \in \mathbb{Z}^+_0} \{ \ang{\MMM_f}{k}^{\dx} \ :\ f \in \SSS\}. \]
Here $D^{-1}$ is isomorphic to $\tilde{d}^{-1}$ which is the inverse of differentiation on the space $\SSS$. Assuming $\SSS$ is closed under integration, therefore, an appropriate inverse of $D$ exists and  the free Lie algebra of multiplicative operators $\MMM(\SSS)$ and polynomials (with constant coefficients) in the differential operators, $\mathcal{P}(\partial_x)$, is characterised completely by $\mathfrak{F}^{\dx}$.

For convenience, we abuse notation and write
\[ \ang{f}{k} = \ang{M_f}{k}^{\dx} = \frac12 \left( M_f \circ \partial_x^k + \partial_x^k \circ  M_f \right). \]
In this notation
\[ \ang{1_{\SSS}}{k} = \partial_x^k, \qquad \ang{f}{0} = M_f, \]
where $1_{\SSS}$ is the constant function over $\SSS$ with value $1$.

\subsubsection{Consequences for exponential splitting schemes}
The semiclassical TDSE \R{eq:TDSE},
\[ \partial_t u = \ii \ve \partial_x^2 u - \ii \ve^{-1} V u, \]
can be written in the form
\[ \partial_t u = \ii (\ve \ang{1}{2}  - \ve^{-1} \ang{V}{0}) u, \]
where the unit in $\ang{1}{2}$ is understood to be $1_\SSS$. We can find the solution of this equation by formally exponentiating the Hamiltonian without discretisation,
\[ u(t) = \exp(\ii t \ve \ang{1}{2} - \ii t \ve^{-1} \ang{V}{0})\ u(0). \]
We may now perform splittings directly on this undiscretised Hamiltonian.

Methods which featured nested matrix commutators such as $[[\DDD_V, \KKK^2], \KKK^2]$
now feature corresponding commutators of operators $\ang{1}{2}$ and $\ang{V}{0}$, such as $[[\ang{V}{0}, \ang{1}{2}], \ang{1}{2}]$. Since the two operators reside in our Lie algebra,
\[ \ang{1}{2}, \ang{V}{0} \in \mathfrak{F} =  \bigoplus_{k \in \mathbb{Z}^+_0}\  \{ \ang{f}{k}\ :\ f \in \SSS\},\]
so do all of their commutators. Moreover, as a consequence of Theorem~\ref{thm:anglie}, commutators in this Lie algebra can be solved explicitly,
\begin{equation}
    \label{eq:angdx}
    \left[\ang{f}{k},\ang{g}{l}\right] = \sum_{n=0}^{\frac{k+l-1}{2}} \sum_{i=0}^{2n+1} \lambda_{n,i}^{k,l} \ang{(\dx^i f) (\dx^{2n+1-i} g)}{k+l-2n-1},
\end{equation}
where $\lambda$s are the same as before.
This allows us to design commutator-free methods by explicitly working out nested commutators such as $[[\ang{V}{0}, \ang{1}{2}], \ang{1}{2}]$
and $[[\ang{V}{0}, \ang{1}{2}], \ang{V}{0}]]$ appearing in the sBCH of $\ii t \ve \ang{1}{2}$ and $-\ii t \ve^{-1} \ang{V}{0}$ using the rules,
    \begin{Eqnarray}
        \nonumber\left[ \Ang{2}{f}, \Ang{1}{g} \right] & =&  - \Ang{0}{(\dx f) (\dx^2 g) + \Frac12 f (\dx^3 g)} + \Ang{2}{2 f (\dx g) - (\dx f) g},\\
        \nonumber\left[ \Ang{2}{f}, \Ang{0}{g} \right] & =& 2 \Ang{1}{f (\dx g)},\\
        \nonumber\left[ \Ang{1}{f}, \Ang{0}{g} \right] & =& \Ang{0}{f (\dx g)},
    \end{Eqnarray}
which can be read off Table~\ref{tab:lambda} in Appendix~\ref{app:tables}. Consequently,
    \begin{Eqnarray}
        \nonumber\left[ \Ang{0}{V}, \Ang{2}{1} \right] & =& -2 \ang{\dx V}{1},\\
        \nonumber[[ \ang{V}{0}, \ang{1}{2}], \ang{V}{0}]] & =& -2 \ang{(\dx V)^2}{0},\\
        \nonumber[[ \ang{V}{0}, \ang{1}{2}], \ang{1}{2}] & =& -\ang{\dx^4 V}{0} + 4 \ang{\dx^2 V}{2}.
    \end{Eqnarray}
The sBCH of  $\ii t \ve \ang{1}{2}$ and $-\ii t \ve^{-1} \ang{V}{0}$ up to $\OO{t^5}$, for instance, is
\[ \ii t \ve \ang{1}{2} -\ii t \ve^{-1} \ang{V}{0} -\Frac{1}{6} \ii t^3 \ve  \ang{\dx^2 V}{2} + \Frac{1}{24} \ii t^3 \ve \ang{\dx^4 V}{0} - \Frac{1}{6} t^3 \ve^{-1} \ang{(\dx V)^2}{0}. \]
Another favourable consequence of Theorem~\ref{thm:anglie} is that all terms in the sBCH of $\ii t \ve \ang{1}{2}$ and $-\ii t \ve^{-1} \ang{V}{0}$
reside in
    \[\mathfrak{H} = \bigoplus_{k \in \mathbb{Z}^+_0}\ \{ \ii^{k+1} \ang{f}{k}\ :\ f\in \SSS \}, \]
whose elements are skew-Hermitian operators. Upon discretisation using pseudospectral methods, these are replaced by
    \[ \ii^{k+1} \ang{f}{k} \rightsquigarrow \Frac12 \ii^{k+1} (\DDD_f \KKK^k + \KKK^k \DDD_f), \]
where $\DDD_f$ is a diagonal matrix and $\KKK$ a skew-symmetric circulant differentiation matrix. Consequently, the discretised forms of elements of $\mathfrak{H}$ are skew-Hermitian matrices,
the exponentials of which are unitary matrices.
Zassenhaus splittings \cite{bader14eaf} feature exponentials of terms in $\mathfrak{H}$, thereby guaranteeing unitary evolution of the wavefunction and resulting in unconditional stability of these numerical methods.


It is well known that solutions of the semiclassical TDSE \R{eq:TDSE}---irrespective of smooth initial conditions---develop oscillations of frequency
$\OO{\ve^{-1}}$ in both space and time \cite{jin11mac}. These oscillations make computational solutions very costly since one typically needs $M = \OO{\ve^{-1}}$
grid points to resolve spatial oscillations and time steps of size $\OO{\ve}$ to resolve temporal oscillations.

The spectral radius of the differentiation matrix $\KKK$ scales as $\OO{M}$, i.e. $\OO{\ve^{-1}}$. Keeping eventual discretisation in mind and noting that
\[ \ang{f}{k} \rightsquigarrow \Frac12  (\DDD_f \KKK^k + \KKK^k \DDD_f) = \OO{\ve^{-k}}, \]
we use the shorthand $\ang{f}{k} = \OO{\ve^{-k}}$. It is now easy to see how the height of a term $a \in \mathfrak{G}_n$
acts as a proxy for the spectral radius of the eventual discretisation,
\begin{equation*}
 a \in \GG{G}_n \Longleftrightarrow \mathrm{ht}(a) \leq n\ \Longrightarrow\ a = \OO{\ve^{-n}}.
\end{equation*}
Since the matrix $\KKK^2$ scales as $\OO{\ve^{-2}}$, one would expect the commutator $[[\DDD_V, \KKK^2], \KKK^2]$
to scale as $\OO{\ve^{-4}}$. The undiscretised commutator $[[\ang{V}{0}, \ang{1}{2}], \ang{1}{2}]$, on the other hand, scales as $\OO{\ve^{-2}}$ thanks to
the property of height reduction (Corollary~\ref{cor:ht}).

\begin{cor}
    \label{cor:ve}
    \normalfont
    Commutators of $A = \ii \ve \ang{1}{2}$ and $B = \ii \ve^{-1} \ang{V}{0}$ are  of size $\OO{\ve^{-1}}$.

    \begin{proof}
    Consider a grade $n$ commutator $C$ of $\ang{1}{2}$ and $\ang{V}{0}$, featuring $k$ occurrences of the letter $\ang{1}{2}$ and $n-k$ occurrences of $\ang{V}{0}$. Since $\ang{1}{2} \in \mathfrak{G}_2$ and $\ang{V}{0} \in \mathfrak{G}_0$, using Corollary~\ref{cor:ht},
    \[ ht(C) \leq 2k - n + 1. \]
     Thus $C$ is $\OO{\ve^{-2k+n-1}}$. However, the corresponding commutator of $A$ and $B$ is scaled by $k$ occurrences of $\ve$ and $n-k$ occurrences of $\ve^{-1}$, bringing its size to $\OO{\ve^{-1}}$.
    \end{proof}
\end{cor}
When we scale the time step as $t = \OO{\ve^{\sigma}}$ for some $\sigma >0$, a grade $n$ commutator of $A = \ii t \ve \ang{1}{2}$ and $B = \ii t \ve^{-1} \ang{V}{0}$,
 scales as $\OO{\ve^{n \sigma -1}}$ in contrast to the $\OO{\ve^{n(\sigma -1)}}$ scaling seen when working with commutators of matrices.
Whereas the latter requires $\sigma > 1$ for convergence of sBCH (which corresponds to very small time steps), our approach does not have such
restrictions and allows convergence for extremely large time steps. With time steps of size $\OO{\sqrt{\ve}}$, which corresponds to $\sigma=1/2$, for instance,
grade $n$ commutators are $\OO{\ve^{n/2-1}}$ and the sBCH series easily converges.

The property of height reduction underlies the asymptotic splitting of the symmetric Zassenhaus kind where it results in inexpensive exponentiations
via Lanczos iterations \cite{bader14eaf}. The advantages of height reduction also extend to error analysis for existing methods such as Yoshida splittings
since commutators discarded in these methods can be assumed to be undiscretised and analysed appropriately.
An order six Yoshida splitting discards all commutators of grade seven and higher, for instance, thereby committing an error of $\OO{\ve^{7\sigma-1}}$.

Despite this, the quadratic cost of the
Zassenhaus splitting, which results from a systematic exploitation of the structures of $\mathfrak{F}$ and the property of height reduction, trumps the exponential cost of Yoshida.

\begin{theorem}Cost of Zassenhaus splittings grows quadratically in the order desired.
\label{thm:quadratic}
\end{theorem}
    \begin{proof} \normalfont
   Zassenhaus splittings \cite{bader14eaf} recursively utilise the sBCH formula featuring odd grade commutators of $\ii t \ve \ang{1}{2}$ and $-\ii t \ve^{-1} \ang{V}{0}$.  Due to Corollary~\ref{cor:ve}, any grade $n$ commutator of $\ii t \ve \ang{1}{2}$ and $-\ii t \ve^{-1} \ang{V}{0}$ is $\OO{\ve^{n\sigma-1}}$.  Splittings with $\OO{\ve^{(2n+3)\sigma -1}}$ error, therefore, require commutators of grades up to  $2n+1$.
    Such a splitting has the form,
    \begin{equation*}
      \label{eq:splitting_FoCM}
      \ee^{\ii t (\ve \ang{1}{2}- \ve^{-1} \ang{V}{0})}=
      \ee^{\frac12 W^{[0]}}\cdots \ee^{\frac12 W^{[n]}}\ee^{W^{[n+1]}} \ee^{\frac12 W^{[n]}}\cdots \ee^{\frac12 W^{[0]}}+\OO{\ve^{(2n+3)\sigma-1}},
    \end{equation*}
    where $W^{[0]}= \ii t \ve \ang{1}{2}$ and $W^{[1]}=-\ii t \ve^{-1} \ang{V}{0}$, both of which are $\OO{\ve^{\sigma-1}}$. For $k \geq 2$, $W^{[k]}$ arises from grade $2k-1$ commutators of $\ii t \ve \ang{1}{2}$ and $-\ii t \ve^{-1} \ang{V}{0}$, and is, therefore, of size $\OO{\ve^{(2k-1)\sigma-1}}$. We refer the reader to \cite{bader14eaf} for details of these splittings.

    The first exponent, $W^{[0]}$, is discretised as a circulant matrix, $\ii t \ve \KKK^2$, whose exponential can be evaluated using two Fast Fourier Transforms (FFTs), each of which costs $\OO{M \log M} = \OO{\ve^{-1} \log(\ve^{-1})}$. The second exponent, $W^{[1]}$, is discretised as a diagonal matrix, $-\ii t \ve^{-1} \DDD_{V}$, and can be exponentiated directly.

    Remaining exponents require Lanczos iterations. Since $W^{[k]}$ is $\OO{\ve^{(2k-1)\sigma-1}}$, we need $\left\lceil \frac{(2n+3)\sigma-1}{(2k-1)\sigma-1} \right\rceil$ iterations for approximating its exponential to $\OO{\ve^{(2n+3)\sigma -1}}$ accuracy \cite{bader14eaf}. Each of these iterations require the evaluation of matrix-vector products of the form $W^{[k]} u$.

    Since $W^{[k]}$ arises from grade $2k-1$ commutators of the terms $\ii t \ve \ang{1}{2} \in \mathfrak{G}_2 \cap \mathfrak{e}$ and $-\ii t \ve^{-1} \ang{V}{0} \in \mathfrak{G}_0 \cap \mathfrak{e}$, it resides in the intersection of $\mathfrak{G}_{2k-2}$ (at most) and $\mathfrak{e}$. Thus, it consists of terms of the form $\ang{f_0}{0},\ang{f_2}{2},$ $\ang{f_4}{4},\ldots,\ang{f_{2k-2}}{2k-2}$. To evaluate $W^{[k]} u$, we need to evaluate $\ang{f_0}{0} u,\ang{f_2}{2} u$, $\ang{f_0}{4} u,\ldots,\ang{f_{2k-2}}{2k-2} u$. The first of these, $\ang{f_0}{0} u$, is a pointwise product, while the rest require four FFTs each since $\ang{f}{k}$ is discretised as $(\DDD_f \KKK^2 + \KKK^2 \DDD_f)/2$.

    The cost of the splitting is dominated by the cost of FFT operations, each of which requires $\OO{\ve^{-1} \log(\ve^{-1})}$ operations. Since each evaluation of $W^{[k]}u$ requires $4(k-1)$ FFTs, the number of FFTs required per time step of the splitting scheme is
    \[ C_{\sigma}(n) = 4 + 2 \sum_{k=2}^{n} 4 (k-1) \left\lceil \frac{(2n+3)\sigma-1}{(2k-1)\sigma-1} \right\rceil +  4 n \left\lceil \frac{(2n+3)\sigma-1}{(2n+1)\sigma-1} \right\rceil ,\]
    which grows quadratically. Under $\sigma = 1$, for instance, this grows as $\sim12n^2$,
    \begin{Eqnarray*}
        C_1(n) &=& 4 + 2 \sum_{k=2}^{n} 4 (k-1) \left\lceil \frac{2n+2}{2k-2} \right\rceil +  4 n \left\lceil \frac{2n+2}{2n} \right\rceil \\
        &\leq & 4+ 8 \sum_{k=2}^{n}(k-1) \left(\frac{2n+2}{2k-2}+1\right)+  8n \\
        &=& 4+ 8 \sum_{k=2}^{n}(n+k)+  8n = 12 n^2 + 4n -4.
    \end{Eqnarray*}
    The overall cost for an $\OO{\ve^{(2n+3)\sigma -1}}$ splitting is $\OO{C_\sigma(n) \ve^{-1} \log{\ve^{-1}}}$ per time step. We remark that careful choices can allow us to reduce the exact number of FFTs required. However, we are largely concerned with asymptotic growths here.

    Since this proof does not assume a specific form of the exponents, $W^{[k]}$s, and works solely on the basis of the structure of $\mathfrak{G}_k$ and $\mathfrak{e}$, the odd graded nature of the sBCH formula and the height reduction of Corollary~\ref{cor:ht}, the cost estimates derived here also apply directly to Magnus--Zassenhaus approaches of \cite{IKS,MKO} discussed in Section~\ref{sec:Vt}.
\end{proof}

\subsubsection{Some related algebras}
The gain of powers of $\ve$ has also been observed by \citeasnoun{lassergaim} using Moyal brackets
in the phase space. However, the analysis is not as easily generalised as the algebraic approach here.

These ideas are closely related to the Weyl algebra \cite{Dixmier1968,CoutinhoWeyl} which is the universal enveloping algebra of the Lie algebra of Heisenberg groups. The univariate Weyl algebra can be written in the form
\[ \mathfrak{W} = \left\{ \sum_{k=0}^n f_k(x) \dx^k \ : \ f_k \in \PPP(x) ,\ k \in \mathbb{Z}^+_0  \right\}. \]
These are special cases of a more general form
\[ \mathfrak{S} = \left\{ \sum_{k=0}^n f_k(x) \dx^k \ : \ f_k \in \SSS ,\ k \in \mathbb{Z}^+_0  \right\}, \]
under the choice of polynomials in $x$, $\PPP(x)$, as the function space $\SSS$.

We note that $\mathfrak{S}$ is contained within $\mathfrak{F}$ as an associative algebra since $\MMM(f_k), \dx^k \in \mathfrak{F}$, while, in the other direction, $\mathfrak{F} \subseteq \mathfrak{S}$ is evident using the Leibniz rule. Thus, the two are identical as associative algebras. It is not difficult to prove that a height reduction phenomenon also holds for $\mathfrak{S}$.

However, there is a significant advantage to working in the symmetrised form of $\mathfrak{F}$ since a $\mathfrak{S}$ lacks a $\mathbb{Z}_2$-graded structure.
The $\mathbb{Z}_2$ grading on $\mathfrak{F}$ and the symmetric nature of its elements ensures that elements of $\mathfrak{H}$ are skew-Hermitian after discretisation. This proves crucial for devising stable numerical schemes such as the symmetric Zassenhaus schemes of \cite{bader14eaf} once we start utilising nested commutators. In contrast, lacking a $\mathbb{Z}_2$ grading, a clean separation of terms does not occur in the form $\mathfrak{S}$.

One could argue that any nested commutator of skew-Hermitian operators in $\mathfrak{S}$ should be skew-Hermitian. This is indeed true prior to discretisation. However, unless some highly specialised differentiation matrices can be constructed, the discretised version of the commutator simplified in $\mathfrak{S}$ possesses no such structure. What is more, even prior to discretisation we prefer to discard terms smaller than a certain size (while analysing in powers of $\ve$). Working in $\mathfrak{S}$ instead of $\mathfrak{F}$, it becomes difficult to discern which components of a term such as $\sum_{k=0}^n f_k(x) \dx^k$ can be discarded and which need to be kept despite their small size in order to preserve the skew-Hermiticity of the undiscretised operator.

\subsection{Time-varying potentials in the TDSE}
\label{sec:Vt}
Consider the TDSE with time-varying electric potential, $V(x,t)$, written in the form
\[ \partial_t u = \AAA(t) u, \]
where $\AAA(t) = \ii \ve \ang{1}{2} - \ii \ve^{-1} \ang{V(t)}{0} \in \mathfrak{H}$. The solution of such Lie group equations can be expressed using the Magnus expansion \cite{magnus54ote},
\[ u(t) = \ee^{\Omega(t)} u(0), \]
where the infinite series $\Omega(t)=\sum_{k=1}^\infty\Omega_k(t)$, is an element of the underlying Lie algebra $\mathfrak{H}$.
The exponent $\Omega(t)$ satisfies the \emph{dexpinv equation} \cite{iserles99ots},
\begin{equation}
\label{eq:dexpinv}
\Omega'(t) = \dexp^{-1}_{\Omega(t)}(\AAA(t)) = \sum_{k=0}^\infty \frac{B_k}{k!} \ad_{\Omega(t)}^k(\AAA(t)), \qquad \Omega(0)=0,
\end{equation}
where $B_k$ are the Bernoulli numbers. The solution of (\ref{eq:dexpinv}), originally proposed as an infinite series using Picard iterations \cite{magnus54ote},
has been widely analysed in  \cite{iserles99ots,iserles00lgm,blanes09tme}.

The graded version of the Magnus expansion \cite{iserles99ots} reads
\begin{align*}\label{eq:SymMag}
\Omega(t)=&\int_0^t\AAA(\xi)d\xi -\frac{1}{2}\int_0^t\int_0^{\xi_1}[\AAA(\xi_2),\AAA(\xi_1)]d\mathbf{\xi}\\
&+\frac{1}{12}\int_0^t\int_0^{\xi_1}\int_0^{\xi_1}[\AAA(\xi_2),[\AAA(\xi_3),\AAA(\xi_1)]]d\mathbf{\xi}\\
&+\frac{1}{4}\int_0^t\int_0^{\xi_1}\int_0^{\xi_2}[[\AAA(\xi_3),\AAA(\xi_2)],\AAA(\xi_1)]d\mathbf{\xi} +\cdots.
\end{align*}
Once again, since $\AAA(t) \in \mathfrak{H}$, all commutators can be simplified to commutator-free terms in $\mathfrak{H}$. This is in contrast to \cite{hochbruck03omi}
where the Magnus expansion features nested commutators of matrices.

Bounds for convergence of the Magnus expansion analysed in \cite{iserles99ots} were improved upon in \cite{moan08}. Applied directly to the semiclassical
TDSE, these would impose a highly stringent restriction of $t=\OO{\ve^2}$ on the time step. \citeasnoun{hochbruck03omi} analysed the convergence of the Magnus
expansion in the context of the TDSE, making significant improvements on the earlier and more general bounds. Nevertheless, under semiclassical scaling their analysis suggests
a need for very small time steps of size $\OO{\ve}$ for convergence.

However, due to Corollary~\ref{cor:ve}, a grade $n$ term in the Magnus expansion of $\AAA(t)$ is $\OO{\ve^{n \sigma -1}}$. Thus, asymptotically speaking in terms of $\ve$, the terms in the expansion are decreasing in size with increasing $n$ for any $\sigma>0$, so
that convergence of the Magnus expansion also occurs for much larger time steps such as $t=\OO{\sqrt{\ve}}$. This is a considerable improvement over earlier analysis.

Different versions of the Magnus expansion based on these observations have been successfully combined with Zassenhaus splittings to yield
efficient methods for solving the TDSE with time-varying potentials \cite{IKS,MKO}. The skew-Hermitian nature of elements of $\mathfrak{H}$ once again leads to unitary evolution of our methods. As noted in Theorem~\ref{thm:quadratic}, the cost of these methods also grows quadratically in the order desired. This should be contrasted against Yoshida based methods---instead of requiring $3$ exponentials, the Strang splitting features $5$ or more terms for the first non-trivial Magnus expansion and consequently the cost of Yoshida methods (which are derived by composing Strang splittings) grows as $\OO{5^n}$ for an order-$2n$ accuracy. This is considerably more expensive than the $\OO{3^n}$ cost for time-independent potentials.

\subsection{A finite dimensional example}
While the examples of the previous sections are based in infinite dimensional algebras and function spaces, it is also possible to construct finite dimensional examples of such structures.
Consider the commutative subalgebra of $2n \times 2n$ matrices, $\AAA = M_{2n}(\mathbb{R})$,
\[ \CCC = \left\{ \left(
                    \begin{array}{cc}
                      a \mathrm{I}_n & A \\
                      \mathrm{O}_n & a \mathrm{I}_n \\
                    \end{array}
                  \right)
  \ :\ a \in \mathbb{R},\ A \in M_{n}(\mathbb{R}) \right\}, \]
where $\mathrm{I}_n$ is the $n \times n$ identity matrix and $\mathrm{O}_n$ is the $n \times n$ zero matrix. Consider any
\[ d = \left(
                    \begin{array}{cc}
                      E^{11} & E^{12} \\
                      \mathrm{O}_n & E^{22} \\
                    \end{array}
                  \right), \]
where $E^{ij} \in M_n(\mathbb{R})$. Then for every $x \in \CCC$,
\[ [d, x] = \left(
                    \begin{array}{cc}
                      \mathrm{O}_n & E^{11} A - A E^{22} \\
                      \mathrm{O}_n & \mathrm{O}_n \\
                    \end{array}
                  \right) \in \CCC. \]
Thus, $d$ is in the Lie idealiser of $\CCC$, $D = ad_d$ is a derivation on $\CCC$, and, consequently,
\[ \mathfrak{F}^d = \bigoplus_{k \in \mathbb{Z}^+_0} \{ \ang{x}{k}^d = \Frac12( x d^k + d^k x)\ :\ x \in \CCC\} \]
is a $\mathbb{Z}_2$-graded matrix Lie algebra.

Similar structures should be found in other commutative subalgebras of $M_n(\mathbb{R})$. Applications of these structural observations to linear algebra algorithms are yet to be explored.

\subsection{Future work}

\begin{enumerate}
\item Significantly relaxed bounds, improving on the bounds
implied by \cite{suzuki77ceo,hochbruck03omi,moan08,casas12} for the convergence of the sBCH, Magnus and Zassenhaus expansions in the case of semiclassical TDSEs
have been obtained due to the property of height reduction. Tighter bounds will require a careful analysis of the growth of the coefficients
$\pi_{n,i}^{k,l}$ and derivatives
of functions in \R{eq:angdx}.

\item The TDSEs analysed here are one-dimensional. Extending these ideas to $\partial_t u = \ii \ve \Delta u - \ii \ve^{-1} V(\MM{x}) u$ will require an analysis
of the free Lie algebra generated by polynomials in the commuting elements of the Lie idealiser, $\dx,\partial_y$ and $\partial_z$.

\item The Wigner equation $\partial_t w(x,\xi,t) = - \xi  \dx w(x,\xi,t) + \Theta^\ve[\delta V]\,w(x,\xi,t)$,
where $\Theta^\ve[\delta  V]$ is a pseudo-differential operator, features operators which are considerably more complicated \cite{jin11mac}.
Yet, the Lie algebraic structure is similar to $\mathfrak{F}$ with the role of the isomorphism played by $\Theta$.
In addition to $\dx$, multiplication by the frequency variable
$\xi$ is also an element of the Lie idealiser. While $\dx$ induces the derivation $\dx$ as usual, $\xi$ induces the derivation $\ii \partial_y$.

\item Matrix-valued potentials appear in TDSEs when we need to consider multiple energy levels at once. Magnus expansion based methods such as \cite{hochbruck03omi} are used once these potentials start featuring time-varying components \cite{karlsson}.

    Unfortunately, matrix-valued potentials do not directly fall into the framework of $\mathfrak{F}$ as a Lie algebra. However suitable extensions for these contexts are being actively sought. Our initial findings suggest that a modified version of the height reduction rule holds and it might be possible to extend the significantly milder time step restrictions and favourable computational complexity of \cite{bader14eaf,IKS,MKO} to these cases.
\end{enumerate}

\subsubsection*{Acknowledgements} The author is grateful to Marjorie Batchelor, Kurusch Ebrahimi--Fard, Dimitar Grantcharov, Arieh Iserles, Caroline Lasser, Hans Munthe--Kaas and Antonella Zanna for guidance on algebraic notions, help with combinatorial techniques and for many helpful discussions.

\bibliographystyle{agsm}
\bibliography{ref}


\appendix

\section{Tables of coefficients}
\label{app:tables}
The coefficients $\pi_{n,i}^{k,l}$ appearing in Lemma~\ref{lem:angass} for $k+l$ ranging between $1$ and $4$ are presented in Table~\ref{tab:pi}. The case $\ang{x}{0} \cdot \ang{y}{0} = \ang{xy}{0}$ is trivial and not listed.
    \begin{table}[h]
    \begin{center}
    \begin{tabular}{|c|c|c|c|c|c|c|}
    \hline
        $(k,l)$ &  $n$ & $\pi_{n,0}^{k,l}$ & $\pi_{n,1}^{k,l}$ & $\pi_{n,2}^{k,l}$ & $\pi_{n,3}^{k,l}$ & $\pi_{n,4}^{k,l}$ \\
        \hline \hline
        $(1,0)$ & $0$ & $1$ &  &  &  &   \\
                & $1$ & $1/2$ & $0$ &  &  &   \\
        \hline
        $(2,0)$ &$0$ & $1$ &  &  &  &  \\
                &$1$ & $1$ & $0$ &  &  &  \\
                &$2$ & $0$ & $-1/2$ & $0$  &  &   \\
        \hline
        $(1,1)$ &$0$ & $1$ &  &  &  &   \\
                &$1$ & $1/2$ & $-1/2$ &  &  &  \\
                &$2$ & $-1/4$ & $-3/4$ & $-1/4$  &  &   \\
        \hline
        $(3,0)$ &$0$ & $1$ &  &  &  &  \\
                &$1$ & $3/2$ & $0$ &  &  &  \\
                &$2$ & $0$ & $-3/2$ & $0$  &  &   \\
                &$3$ & $-1/4$ & $-3/4$ & $0$  & $0$ &  \\
        \hline
        $(2,1)$ &$0$ & $1$ &  &  &  &  \\
                &$1$ & $1$ & $-1/2$ &  &  & \\
                &$2$ & $-1/2$ & $2$ & $-1/2$  &  &  \\
                &$3$ & $-1/4$ & $-1/2$ & $0$  & $0$ & \\
        \hline
        $(4,0)$ &$0$ & $1$ &  &  &  & \\
                &$1$ & $2$ & $0$ &  &  &  \\
                &$2$ & $0$ & $-3$ & $0$  &  &  \\
                &$3$ & $-1$ & $-3$ & $0$  & $0$ & \\
                &$4$ & $0$ & $1/2$ & $3/2$  & $1/2$ & $0$ \\
        \hline
        $(3,1)$ &$0$ & $1$ &  &  &  & \\
                &$1$ & $3/2$ & $-1/2$ &  &  &  \\
                &$2$ & $-3/4$ & $-15/4$ & $-3/4$  &  &  \\
                &$3$ & $-1$ & $-9/4$ & $0$  & $0$ & \\
                &$4$ & $1/8$ & $1$ & $15/8$  & $7/8$ & $1/8$ \\
        \hline
        $(2,2)$ &$0$ & $1$ &  &  &  & \\
                &$1$ & $1$ & $-1$ &  &  &  \\
                &$2$ & $-1$ & $-4$ & $-1$  &  &  \\
                &$3$ & $-1/2$ & $-1$ & $1$  & $1/2$ & \\
                &$4$ & $1/4$ & $5/4$ & $9/4$  & $5/4$ & $1/4$ \\
    \hline
    \end{tabular}
    \caption{A table of the coefficients $\pi_{n,i}^{k,l},\ n \in \{0,\ldots, k+l\},\ i \in \{0,\ldots,n\}$, which appear in Lemma~\ref{lem:angass}. }
    \label{tab:pi}
    \end{center}
    \end{table}
From Lemma~\ref{lem:symmetry_pi} we know that $\pi_{n,i}^{k,l} = (-1)^n \pi_{n,n-i}^{l,k}$, so that specifying the rows $(1,3)$ and $(0,4)$, for instance, would be redundant.


In Table~\ref{tab:lambda} we present the coefficients $\lambda_{n,i}^{k,l}$ appearing in Theorem~\ref{thm:anglie} for $k+l$ ranging between $1$ and $6$, while noting that the relation $\lambda_{n,i}^{k,l} = - \lambda_{n,2n+1-i}^{l,k}$ makes redundant the need to specify coefficients when $k$ and $l$ exchange values.
    \begin{table}[h]
    \begin{center}
    \begin{tabular}{|c|c|c|c|c|c|c|c|}
    \hline
        $(k,l)$ &  $n$ & $\lambda_{n,0}^{k,l}$ & $\lambda_{n,1}^{k,l}$ & $\lambda_{n,2}^{k,l}$ & $\lambda_{n,3}^{k,l}$ & $\lambda_{n,4}^{k,l}$ & $\lambda_{n,5}^{k,l}$\\
        \hline \hline
        $(1,0)$ & $0$ & $1$ & $0$ &  &  &  & \\
        \hline
        $(2,0)$ &$0$ & $2$ & $0$ &  &  &  & \\
        \hline
        $(1,1)$ & $0$ & $1$ & $-1$ &  &  &  & \\
        \hline
        $(3,0)$ &  $0$ & $3$ & $0$ &  &  &  & \\
                &  $1$ & $-1/2$ & $-3/2$ & $0$  & $0$  &  & \\
        \hline
        $(2,1)$ &  $0$ & $2$ & $-1$ &  &  &  & \\
                &  $1$ & $-1/2$ & $-1$ & $0$  & $0$  &  & \\
        \hline
        $(4,0)$ &  $0$ & $4$ & $0$ &  &  &  & \\
                &  $1$ & $-2$ & $6$ & $0$  & $0$  &  & \\
        \hline
        $(3,1)$ &  $0$ & $3$ & $-1$ &  &  &  & \\
                &  $1$ & $-2$ & $-9/2$ & $0$  & $0$  &  & \\
        \hline
        $(2,2)$ &  $0$ & $2$ & $-2$ &  &  &  & \\
                &  $1$ & $-1$ & $-2$ & $2$  & $1$  &  & \\
        \hline
        $(5,0)$ &  $0$ & $5$ & $0$ &  &  &  & \\
                &  $1$ & $-5$ & $-15$ & $0$  & $0$  &  & \\
                &  $2$ & $1$ & $5$ & $15/2$  & $5/2$  & $0$  & $0$  \\
        \hline
        $(4,1)$ &  $0$ & $4$ & $-1$ &  &  &  & \\
                &  $1$ & $-5$ & $-12$ & $0$  & $0$  &  & \\
                &  $2$ & $1$ & $9/2$ & $6$  & $2$  & $0$  & $0$  \\
        \hline
        $(3,2)$ &  $0$ & $3$ & $-2$ &  &  &  & \\
                &  $1$ & $-7/2$ & $-15/2$ & $3$  & $3/2$  &  & \\
                &  $2$ & $3/4$ & $3$ & $7/2$  & $0$  & $-1$  & $-1/4$  \\
        \hline
        $(6,0)$ &  $0$ & $6$ & $0$ &  &  &  & \\
                &  $1$ & $-10$ & $-30$ & $0$  & $0$  &  & \\
                &  $2$ & $6$ & $30$ & $45$  & $15$  & $0$  & $0$  \\
        \hline
        $(5,1)$ &  $0$ & $5$ & $-1$ &  &  &  & \\
                &  $1$ & $-10$ & $-25$ & $0$  & $0$  &  & \\
                &  $2$ & $6$ & $55/2$ & $75/2$  & $25/2$  & $0$  & $0$  \\
        \hline
        $(4,2)$ &  $0$ & $4$ & $-2$ &  &  &  & \\
                &  $1$ & $-8$ & $-18$ & $4$  & $2$  &  & \\
                &  $2$ & $5$ & $21$ & $26$  & $4$  & $-4$  & $-1$  \\
        \hline
        $(3,3)$ &  $0$ & $3$ & $-3$ &  &  &  & \\
                &  $1$ & $-5$ & $-21/2$ & $21/2$  & $5$  &  & \\
                &  $2$ & $3$ & $12$ & $21/2$  & $-21/2$  & $-12$  & $-3$  \\
    \hline
    \end{tabular}
    \caption{A table of the coefficients $\lambda_{n,i}^{k,l},\ n \in \{0,\ldots, (k+l-1)/2\},\ i \in \{0,\ldots,2n+1\}$, which appear in Theorem~\ref{thm:anglie}. }
    \label{tab:lambda}
    \end{center}
    \end{table}
The values of the coefficients for $(k,l)=(0,2)$ can be inferred from the row $(k,l)=(2,0)$, for instance. Since $\lambda_{n,i}^{k,l} = 2 \pi_{2n+1,i}^{k,l}$, the first eight rows can be read directly by doubling the corresponding rows in Table~\ref{tab:pi}. Note that $[\ang{x}{0}, \ang{y}{0}]=0$, and the case $k+l=0$ doesn't merit a mention in the table. Using the row $(k,l)=(3,2)$ we can compute
    \begin{Eqnarray*}
        \left[ \ang{x}{3}, \ang{y}{2} \right] &=& \ang{3\, x\, D y - 2\, D x\, y}{0} \\
        && + \ang{-(7/2)\, x\, D^3 y -(15/2)\,D x\, D^2 y + 3\, D^2 x\, D y + (3/2)\, D^3 x\ y}{1} \\
        && + \ang{(3/4)\, x\, D^5 y + 3\, D x\, D^4 y + (7/2)\, D^2 x\, D^3 y -  D^4 x\, D y -(1/4)\,  D^5 x\ y}{2}.
    \end{Eqnarray*}
Similarly, in the context of the function space $\SSS$ and differential operator $d = \dx$,
    \[ \left[ \Ang{2}{f}, \Ang{1}{g} \right] = \Ang{0}{- \Frac12 f (\dx^3 g) - (\dx f) (\dx^2 g)} + \Ang{2}{2 f (\dx g) - (\dx f) g}  \]
is found by substituting in \R{eq:angdx} the coefficients from row $(k,l)=(2,1)$.
Note that these brackets are linear,
$\ang{\alpha x + \beta y}{k} = \alpha \ang{x}{k} + \beta \ang{y}{k}$.

\section{Proof of Lemma~\ref{lem:explicit_pi}---explicit form of $\pi_{n,i}^{k,l}$}
\label{app:explicit_pi}
    Let
    \begin{equation}
        \label{eq:S}
        S_{(a,p),(n,i)}^q =  \begin{cases} \delta_{a,n} \delta_{p,i} + \binom{q-n}{a-n} \binom{a-n}{p-i} & a \geq n, p \geq i,\\
                            0 & \mathrm{otherwise}. \end{cases}
    \end{equation}
    The explicit form of  $\pi$s in \R{eq:explicit_pi} allows us to express $R_{a,p}^{k,l}$ in \R{eq:RHSlambda} as
    \begin{Eqnarray*}
    R_{a,p}^{k,l} &=& 2 \sum_{n=0}^{k+l} \sum_{i=0}^{k+l} S_{(a,p),(n,i)}^{k+l} \pi_{n,i}^{k,l}\\
    &=& \sum_{n=0}^{k+l} \sum_{i=0}^{k+l} S_{(a,p),(n,i)}^{k+l} \sum_{s=0}^{k+l} \sum_{j=0}^{k+l} A_{(n,i),(s,j)}^{k+l} L_{s,j}^{k,l}\\
    &=& \sum_{s=0}^{k+l} \sum_{j=0}^{k+l} \left[\sum_{n=0}^{k+l} \sum_{i=0}^{k+l} S_{(a,p),(n,i)}^{k+l} A_{(n,i),(s,j)}^{k+l}\right] L_{s,j}^{k,l}.
    \end{Eqnarray*}
    To prove Lemma~\ref{lem:explicit_pi}, we need to prove that $A$ in \R{eq:explicit_pi} is such that $R_{a,p}^{k,l} =  L_{a,p}^{k,l}$ is satisfied. This certainly holds if
    \begin{equation}
        \label{eq:TPT}
        \sum_{n=0}^{q} \sum_{i=0}^{q} S_{(a,p),(n,i)}^{q} A_{(n,i),(s,j)}^{q} = \delta_{a,s} \delta_{p,j},
    \end{equation}
    holds for any $q$ and $a,s,p,j \in \{0,\ldots,q\}$. To prove this we note that $S$ and $A$, and therefore their product $SA$, are lower triangular. Thus we may concern ourselves solely with the case $0 \leq s \leq a \leq q$ and $0 \leq j \leq p \leq q$. Denoting $(SA)^q_{(a,p),(s,j)}$ as $T$ for brevity,
    \begin{equation*}
    T = \sum_{n=0}^{q} \sum_{i=0}^{q} \left[ \delta_{a,n} \delta_{p,i} + \binom{q-n}{a-n} \binom{a-n}{p-i} \right] \left[ \delta_{n,s}\delta_{i,j} - \frac{P_{n-s+1}}{n-s+1} \binom{q-s}{n-s} \binom{n-s}{i-j} \right]
    \end{equation*}
    can be separated into four parts,
    \begin{Eqnarray*}
    T_1 & = & \sum_{n=0}^{q} \sum_{i=0}^{q} \delta_{a,n} \delta_{p,i} \delta_{n,s}\delta_{i,j} = \delta_{a,s} \delta_{p,j},\\
    T_2 & = & - \sum_{n=0}^{q} \sum_{i=0}^{q} \delta_{a,n} \delta_{p,i} \frac{P_{n-s+1}}{n-s+1} \binom{q-s}{n-s} \binom{n-s}{i-j}\\
        &=& -\frac{P_{a-s+1}}{a-s+1} \binom{q-s}{a-s} \binom{a-s}{p-j},\\
    T_3 & = & \sum_{n=0}^{q} \sum_{i=0}^{q} \delta_{n,s}\delta_{i,j} \binom{q-n}{a-n} \binom{a-n}{p-i} = \binom{q-s}{a-s} \binom{a-s}{p-j},\\
    T_4 & =& - \sum_{n=0}^{q} \sum_{i=0}^{q} \binom{q-n}{a-n} \binom{a-n}{p-i} \frac{P_{n-s+1}}{n-s+1} \binom{q-s}{n-s} \binom{n-s}{i-j}.
    \end{Eqnarray*}
    In the case of $T_4$ we note that the binomial coefficients vanish except for $s \leq n \leq a$ and $j \leq i \leq p$, when expanding them
    leads to the expression
    \[ T_4 = - \frac{(q-s)!}{(q-a)!} \sum_{n=s}^{a} \frac{P_{n-s+1}}{n-s+1} \sum_{i=j}^{p} \frac{1}{(p-i)!(a-n-p+i)!(i-j)!(n-s-i+j)!}. \]
    In order to reduce this expression, will need the following identities,
    \begin{equation}
        \label{eq:1}
        \sum_{i=j}^{p} \frac{1}{(p-i)!(a-n-p+i)!(i-j)!(n-s-i+j)!} = \frac{1}{(a-n)!(n-s)!} \binom{a-s}{p-j},
    \end{equation}
    \begin{equation}
        \label{eq:3}
        \sum_{n=0}^{b} \binom{b+1}{n+1}  (2^{n+1}-1) B_{n+1} = -\frac12 \delta_{b,0} - \delta_{b > 0} P_{b+1},
    \end{equation}
    and
    \begin{equation}
        \label{eq:2}
        \sum_{n=0}^{b} \frac{P_{n+1}}{(n+1)!(b-n)!} =  \frac{1}{b!} - \frac12 \delta_{b,0} - \frac{P_{b+1}}{(b+1)!} \delta_{b > 0},
    \end{equation}
    where $\delta_{b>0}$ is $1$ if $b>0$ and $0$ otherwise. Using these identities, which we prove at the end of this appendix,
    \begin{Eqnarray*}
    T_4 &\overset{\R{eq:1}}{=}&  - \frac{(q-s)!}{(q-a)!} \binom{a-s}{p-j} \sum_{n=s}^{a} \frac{P_{n-s+1}}{(n-s+1)!(a-n)!}\\
    &=&  - \frac{(q-s)!}{(q-a)!} \binom{a-s}{p-j} \sum_{n=0}^{a-s} \frac{P_{n+1}}{(n+1)!((a-s)-n)!}\\
    &\overset{\R{eq:2}}{=}& - \frac{(q-s)!}{(q-a)!} \binom{a-s}{p-j} \left[\frac{1}{(a-s)!} - \frac12 \delta_{a,s} - \frac{P_{a-s+1}}{(a-s+1)!} \delta_{a > s}\right]\\
    &=& -\binom{q-s}{a-s}\binom{a-s}{p-j} + \frac12 \delta_{a,s} \delta_{p,j} +  \frac{P_{a-s+1}}{a-s+1} \binom{q-s}{a-s} \binom{a-s}{p-j} \delta_{a > s}
    \end{Eqnarray*}
    This allows us to complete the proof of Lemma~\ref{lem:explicit_pi} by proving \R{eq:TPT},
    \begin{Eqnarray*}
    T &=& \delta_{a,s} \delta_{p,j} -\frac{P_{a-s+1}}{a-s+1} \binom{q-s}{a-s} \binom{a-s}{p-j} +  \binom{q-s}{a-s} \binom{a-s}{p-j} \\ &&-\binom{q-s}{a-s}\binom{a-s}{p-j} + \frac12 \delta_{a,s} \delta_{p,j} +  \frac{P_{a-s+1}}{a-s+1} \binom{q-s}{a-s} \binom{a-s}{p-j} \delta_{a > s}\\
    &=& \frac32 \delta_{a,s} \delta_{p,j} - (1-\delta_{a>s}) \frac{P_{a-s+1}}{a-s+1} \binom{q-s}{a-s} \binom{a-s}{p-j}\\
    &=& \frac32 \delta_{a,s} \delta_{p,j} - \delta_{a\leq s} \frac{P_{a-s+1}}{a-s+1} \binom{q-s}{a-s} \binom{a-s}{p-j}\\
    &=& \frac32 \delta_{a,s} \delta_{p,j} - \delta_{a,s} \delta_{p,j} P_1 = \delta_{a,s} \delta_{p,j},\\
    \end{Eqnarray*}
    in proving which we have used $P_1=1/2$ and the fact that $SA$ is lower triangular (thus the only case of $a \leq s$ that
    we need to consider is $s=a$). This completes the proof of Lemma~\ref{lem:explicit_pi}.

    Proofs for the identities \R{eq:1}, \R{eq:3} and \R{eq:2} are given below.\\
    \begin{proof}\R{eq:1}
        \begin{Eqnarray*}
        &&\sum_{i=j}^{p} \frac{(a-n)!(n-s)!}{(p-i)!(a-n-p+i)!(i-j)!(n-s-i+j)!} =\sum_{i=j}^p \binom{a-n}{p-i} \binom{n-s}{i-j} \\
        && \qquad =  \sum_{i=0}^{p-j} \binom{a-n}{(p-j)-i} \binom{n-s}{i} =  \binom{a-s}{p-j},
        \end{Eqnarray*}
        since
        \[ \sum_{i=0}^k \binom{n}{k-i} \binom{m}{i} = \binom{n+m}{k}. \]
    \end{proof}\\
    \begin{proof}\R{eq:3}
        \begin{Eqnarray}
        \nonumber \sum_{n=0}^{b} \binom{b+1}{n+1}  (2^{n+1}-1) B_{n+1} &=& \sum_{n=0}^{b} \binom{b+1}{n}  (2^{b-n+1}-1) B_{b-n+1}\\
        \nonumber &=& \sum_{n=0}^{b+1} \binom{b+1}{n}  (2^{b-n+1}-1) B_{b-n+1} - 0 \\
        \nonumber &=& \left[ 2^{b+1} \BBB_{b+1}(1/2) - \BBB_{b+1}(1) \right] \\
        \label{eq:3exp}&=& \left[ 2- 2^{b+1}  - (-1)^{b+1} \right] B_{b+1},
        \end{Eqnarray}
        where $\BBB_k(x)$ are the Bernoulli polynomials,
        \[ \BBB_k(x) = \sum_{i=0}^{k} \binom{k}{i} B_{k-i}\, x^i, \]
        whose values at $1$ and $1/2$ are, \[ \BBB_{k}(1) = (-1)^k B_k,\qquad \BBB_{k}(1/2) = (2^{1-k} - 1)B_k.\]
        For $b=0$, the expression \R{eq:3exp} evaluates to $-1/2$, while for $b>0$, $(-1)^{b+1}$ can be replaced by $1$
        since $B_{b+1}$ vanishes for all cases when $(-1)^{b+1}$ is negative. Thus, for $b>0$, \R{eq:3exp} evaluates to $\left( 1- 2^{b+1}  \right) B_{b+1}$. Using the same logic, we may multiply it by $(-1)^{b+1}$ to get $-P_{b+1}$, completing the proof.
    \end{proof}\\
    \begin{proof}\R{eq:2}
        \begin{Eqnarray*}
            \sum_{n=0}^{b} \frac{P_{n+1}}{(n+1)!(b-n)!} &=& \frac{1}{(b+1)!} \sum_{n=0}^{b} \binom{b+1}{n+1} (-1)^{n+1} (2^{n+1}-1) B_{n+1}\\
            &=& \frac{1}{(b+1)!} \left[ \sum_{n=0}^{b} \binom{b+1}{n+1}  (2^{n+1}-1) B_{n+1} - 2 (b+1) B_1\right]\\
            &=& \frac{1}{b!} + \frac{1}{(b+1)!}  \sum_{n=0}^{b} \binom{b+1}{n+1}  (2^{n+1}-1) B_{n+1}\\
            &\overset{\R{eq:3}}{=}& \frac{1}{b!} - \frac12 \delta_{b,0} - \frac{P_{b+1}}{(b+1)!} \delta_{b > 0}.
        \end{Eqnarray*}
    Where we have used the fact that, except for the $n=0$ case, all negative occurrences of $(-1)^{n+1}$ vanish since $B_{n+1}$ vanishes.
    \end{proof}

\section{Proof of Lemma~\ref{lem:genfun_pi}---generating function}
\label{app:genfun_pi}
We wish to find an explicit form for the generating function
    \begin{equation}
        \tag{\ref{eq:genfun_pi}}
        h(u,w,y,x) = \sum_{l=0}^{\infty} \frac{u^l}{l!} \sum_{k=0}^{\infty} \frac{w^k}{k!} \sum_{n=0}^{k+l} (k+l-n)! y^n \sum_{i=0}^{n} x^i \pi^{k,l}_{n,i}.
    \end{equation}
We start from the result of Lemma~\ref{lem:explicit_pi}, substituting (\ref{eq:explicit_pi}) in \R{eq:genfun_pi} and splitting $\pi_{n,i}^{k,l}$ into eight parts for convenience,
    \begin{equation}
        \label{eq:expansion_pi}
        \pi_{n,i}^{k,l} = \frac12 \sum_{j=1}^8 p_j(k,l,n,i),
    \end{equation}
where
    \begin{Eqnarray}
    \nonumber p_1(k,l,n,i) &=& \delta_{i,0} \binom{k}{n}, \\
    \nonumber p_2(k,l,n,i) &=& \delta_{i,0} \binom{k+l}{n},\\
    \nonumber p_3(k,l,n,i) &=& \binom{k}{i}\binom{k-i}{n-i},\\
    \nonumber p_4(k,l,n,i) &=& \binom{k}{i}\binom{k+l-i}{n-i},\\
    \nonumber p_5(k,l,n,i) &=& - \sum_{r=0}^n \frac{P_{r+1}}{r+1} \binom{k+l-n+r}{r} \binom{r}{i} \binom{k}{n-r},\\
    \nonumber p_6(k,l,n,i) &=& - \sum_{r=0}^n \frac{P_{r+1}}{r+1} \binom{k+l-n+r}{r} \binom{r}{i} \binom{k+l}{n-r},\\
    \nonumber p_7(k,l,n,i) &=& - \sum_{r=0}^n \sum_{j=0}^{n-r} \frac{P_{r+1}}{r+1} \binom{k+l-n+r}{r} \binom{r}{i-j} \binom{k}{j} \binom{k-j}{n-r-j},\\
    \nonumber p_8(k,l,n,i) &=& - \sum_{r=0}^n \sum_{j=0}^{n-r} \frac{P_{r+1}}{r+1} \binom{k+l-n+r}{r} \binom{r}{i-j} \binom{k}{j} \binom{k+l-j}{n-r-j},
    \end{Eqnarray}
are obtained after a change of variables, $r=n-s$, and noting the facts that $\binom{0}{i-j} = \delta_{i,j}$ and that $\binom{n}{k}$ vanishes for $k<0$ when $n \geq 0$.
We will simplify the corresponding parts of $h(u,w,y,x)$,
    \[ h_j(u,w,y,x) = \sum_{l=0}^{\infty} \frac{u^l}{l!} \sum_{k=0}^{\infty} \frac{w^k}{k!} \sum_{n=0}^{k+l} (k+l-n)! y^n \sum_{i=0}^{n} x^i p_j(k,l,n,i), \]
combining them to find an expression for the generating function $h =  \sum_{j=1}^8 h_j/2$. In this pursuit, we will repeatedly use a few results,
    \begin{equation}
        \label{eq:genfun_Br}
        \sum_{r=0}^{\infty} \frac{B_{r}}{r!} x^r = \frac{x}{\ee^x -1},
    \end{equation}
which is well known,
    \begin{Eqnarray}
        \nonumber \sum_{r=0}^{\infty} \frac{P_{r+1}}{(r+1)!} x^r &=& \frac{1}{x} \sum_{r=0}^{\infty} \frac{P_r}{r!} x^r =  \frac{1}{x} \sum_{r=0}^{\infty} \frac{(-1)^r (2^r - 1) B_r}{r!} x^r\\
        \nonumber & = & \frac{1}{x} \sum_{r=0}^{\infty} \frac{B_r }{r!}(-2x)^r  - \frac{1}{x} \sum_{r=0}^{\infty} \frac{B_r }{r!}(-x)^r\\
        \label{eq:genfun_Pr} & = & -\left( \frac{2}{\ee^{-2x} +1} - \frac{1}{\ee^{-x} +1}\right) = \frac{\ee^x}{\ee^x +1},
    \end{Eqnarray}
where we use the fact that $P_0=0$, and
    \begin{Eqnarray}
        \nonumber \sum_{l=0}^{\infty} \sum_{k=0}^{\infty} \binom{k+r}{l} u^l w^{k+n-l} & = &  \sum_{k=0}^{\infty} w^{k+n} \sum_{l=0}^{\infty}\binom{k+r}{l} \left(\frac{u}{w}\right)^l   =   \sum_{k=0}^{\infty} w^{k+n}  \left(1+\frac{u}{w}\right)^k+r  \\
        \label{eq:genfun_wu} & = &   w^{n-r} (w+u)^r \sum_{k=0}^{\infty}  (w+u)^k  =\frac{w^{n-r} (w+u)^r}{1-(w+u)}.
    \end{Eqnarray}
Two standard tricks for exchanging summations that we exploit are
    \begin{Eqnarray}
        \label{eq:sum_nr} \sum_{n=0}^{\infty} \sum_{r=0}^{n} \alpha_{n,r} &=& \sum_{r=0}^{\infty} \sum_{n=0}^{\infty}  \alpha_{n+r,r}, \\
        \label{eq:sum_nk} \sum_{k=0}^{\infty} \sum_{n=0}^{k+l} \alpha_{n,k,l} &=& \sum_{n=0}^{\infty} \sum_{k=0}^{\infty} \alpha_{n,k+n-l,l}.
    \end{Eqnarray}
With these tools, we proceed to seek expressions for the generating functions $h_j$, writing $h_j$ as shorthand for $h_j(u,w,y,x)$.
    \begin{Eqnarray*}
        h_1 & = & \sum_{l=0}^{\infty} \frac{u^l}{l!} \sum_{k=0}^{\infty} \frac{w^k}{k!} \sum_{n=0}^{k+l} (k+l-n)! y^n \sum_{i=0}^{n} x^i \delta_{i,0} \binom{k}{n} \\
        & = & \sum_{l=0}^{\infty} u^l \sum_{k=0}^{\infty} w^k \sum_{n=0}^{k+l} \binom{k+l-n}{l}  \frac{y^n}{n!} =  \sum_{n=0}^{\infty} \frac{y^n}{n!} \left[\sum_{l=0}^{\infty} \sum_{k=0}^{\infty} \binom{k}{l} u^l  w^{k+n-l}   \right] \\
        & = & \frac{1}{1-(w+u)} \sum_{n=0}^{\infty} \frac{(yw)^n}{n!} = \frac{\ee^{yw}}{1-(w+u)},
    \end{Eqnarray*}
    \begin{Eqnarray*}
        h_2 & = & \sum_{l=0}^{\infty} \frac{u^l}{l!} \sum_{k=0}^{\infty} \frac{w^k}{k!} \sum_{n=0}^{k+l} (k+l-n)! y^n \sum_{i=0}^{n} x^i \delta_{i,0} \binom{k+l}{n} \\
        & = & \sum_{l=0}^{\infty} u^l \sum_{k=0}^{\infty} w^k \sum_{n=0}^{k+l} \binom{k+l}{l}  \frac{y^n}{n!} = \sum_{n=0}^{\infty} \frac{y^n}{n!} \left[\sum_{l=0}^{\infty} \sum_{k=0}^{\infty} \binom{k+n}{l} u^l  w^{k+n-l}   \right] \\
        & = & \frac{1}{1-(w+u)} \sum_{n=0}^{\infty} \frac{(y(w+u))^n}{n!} = \frac{\ee^{y(w+u)}}{1-(w+u)},
    \end{Eqnarray*}
    \begin{Eqnarray*}
        h_3 & = & \sum_{l=0}^{\infty} \frac{u^l}{l!} \sum_{k=0}^{\infty} \frac{w^k}{k!} \sum_{n=0}^{k+l} (k+l-n)! y^n \sum_{i=0}^{n} x^i \binom{k}{i}\binom{k-i}{n-i} \\
        & = & \sum_{l=0}^{\infty} u^l \sum_{k=0}^{\infty} w^k \sum_{n=0}^{k+l} \frac{y^n}{n!}  \binom{k+l-n}{l} \left[\sum_{i=0}^n \binom{n}{i} x^i \right]\\
        & = & \sum_{n=0}^{\infty} \frac{(y(1+x))^n}{n!}  \left[ \sum_{l=0}^{\infty}  \sum_{k=0}^{\infty} \binom{k}{l} u^l  w^{k+n-l} \right] \\
        & = & \frac{1}{1-(w+u)} \sum_{n=0}^{\infty} \frac{(yw(1+x))^n}{n!} = \frac{\ee^{yw(1+x)}}{1-(w+u)},
    \end{Eqnarray*}
    \begin{Eqnarray*}
        h_4 & = & \sum_{l=0}^{\infty} \frac{u^l}{l!} \sum_{k=0}^{\infty} \frac{w^k}{k!} \sum_{n=0}^{k+l} (k+l-n)! y^n \sum_{i=0}^{n} x^i  \binom{k}{i}\binom{k+l-i}{n-i} \\
        & = & \sum_{l=0}^{\infty} u^l \sum_{k=0}^{\infty} w^k \sum_{n=0}^{k+l} \frac{y^n}{n!} \sum_{i=0}^n \binom{n}{i} x^i   \binom{k+l-i}{l}  \\
        & = & \sum_{n=0}^{\infty} \frac{y^n}{n!} \sum_{i=0}^n \binom{n}{i} x^i  \left[ \sum_{l=0}^{\infty}  \sum_{k=0}^{\infty} \binom{k+n-i}{l} u^l  w^{k+n-l} \right] \\
        & = &  \frac{1}{1-(w+u)} \sum_{n=0}^{\infty} \frac{y^n}{n!} \sum_{i=0}^n \binom{n}{i} x^i  w^i (w+u)^{n-i} \\
        & = &  \frac{1}{1-(w+u)} \sum_{n=0}^{\infty} \frac{(y((w+u)+xw))^n}{n!} =  \frac{\ee^{y(w+u)+xyw}}{1-(w+u)},
    \end{Eqnarray*}
    \begin{Eqnarray*}
    h_5 & = & - \sum_{l=0}^{\infty} \frac{u^l}{l!} \sum_{k=0}^{\infty} \frac{w^k}{k!} \sum_{n=0}^{k+l} (k+l-n)! y^n \sum_{i=0}^{n} x^i \sum_{r=0}^n \frac{P_{r+1}}{r+1} \binom{k+l-n+r}{r} \binom{r}{i} \binom{k}{n-r} \\
    & = & - \sum_{l=0}^{\infty} u^l \sum_{k=0}^{\infty} w^k \sum_{n=0}^{k+l} y^n \sum_{r=0}^n \sum_{i=0}^{r} x^i  \frac{P_{r+1}}{(r+1)!(n-r)!} \binom{k+l-n+r}{l} \binom{r}{i}\\
    & = & -  \sum_{n=0}^{\infty} y^n \sum_{r=0}^n  \frac{P_{r+1}}{(r+1)!(n-r)!} \left[\sum_{l=0}^{\infty} \sum_{k=0}^{\infty} \binom{k+r}{l} u^l w^{k+n-l} \right] \left[\sum_{i=0}^{r}  \binom{r}{i} x^i \right],
    \end{Eqnarray*}
where we have changed summation limits on $n$ using (\ref{eq:sum_nk}), and for $i$ since $r\leq n$ and $\binom{r}{i}=0$ for $i>r\geq 0$. Using (\ref{eq:genfun_wu}) and changing limits again using (\ref{eq:sum_nr}),
    \begin{Eqnarray*}
    h_5 & = & - \frac{1}{1-(w+u)} \sum_{n=0}^{\infty} y^n \sum_{r=0}^n \frac{P_{r+1}}{(r+1)!(n-r)!} w^{n-r} ((w+u)(1+x))^r \\
    & = & - \frac{1}{1-(w+u)} \left[\sum_{n=0}^{\infty} \frac{(yw)^n}{n!}\right]  \sum_{r=0}^{\infty} \frac{P_{r+1}}{(r+1)!} (y(w+u)(1+x))^r   \\
    & = & - \frac{\ee^{yw}}{1-(w+u)} \frac{\ee^{z}}{\ee^{z}+1},
    \end{Eqnarray*}
where we have defined $z=y(w+u)(1+x)$ for convenience.
    \begin{Eqnarray*}
        h_6 & = & - \sum_{l=0}^{\infty} \frac{u^l}{l!} \sum_{k=0}^{\infty} \frac{w^k}{k!} \sum_{n=0}^{k+l} (k+l-n)! y^n \sum_{i=0}^{n} x^i \sum_{r=0}^n \frac{P_{r+1}}{r+1} \binom{k+l-n+r}{r} \binom{r}{i} \binom{k+l}{n-r} \\
        & = & - \sum_{l=0}^{\infty} u^l \sum_{k=0}^{\infty} w^k \sum_{n=0}^{k+l} y^n \sum_{r=0}^n  \frac{P_{r+1}}{(r+1)!(n-r)!}  \binom{k+l}{l} \left[\sum_{i=0}^{r} \binom{r}{i} x^i\right] \\
        & = & - \sum_{n=0}^{\infty} y^n  \sum_{r=0}^n  \frac{P_{r+1}}{(r+1)!(n-r)!} (1+x)^r \left[\sum_{l=0}^{\infty} \sum_{k=0}^{\infty} \binom{k+n}{l} u^l w^{k+n-l}\right]\\
        & = & - \frac{1}{1-(w+u)}\sum_{n=0}^{\infty} (y(w+u))^n  \sum_{r=0}^n  \frac{P_{r+1}}{(r+1)!(n-r)!} (1+x)^r  \\
        & = & - \frac{1}{1-(w+u)} \left[\sum_{n=0}^{\infty} \frac{(y(w+u))^n}{n!} \right] \sum_{r=0}^\infty  \frac{P_{r+1}}{(r+1)!} (y(w+u)(1+x))^r\\
        & = & - \frac{\ee^{y(w+u)}}{1-(w+u)} \frac{\ee^{z}}{\ee^{z}+1}.
    \end{Eqnarray*}
After manipulating binomial coefficients, we change summation limits (\ref{eq:sum_nk}) and use (\ref{eq:genfun_wu}),
    \begin{Eqnarray*}
        h_7 & = & - \sum_{l=0}^{\infty} u^l \sum_{k=0}^{\infty} w^k \sum_{n=0}^{k+l} y^n  \sum_{r=0}^n \sum_{j=0}^{n-r} \frac{P_{r+1}}{(r+1)! j! (n-r-j)!} \binom{k+l-n+r}{l} \left[\sum_{i=0}^{n} \binom{r}{i-j} x^i \right]  \\
        & = & -  \sum_{n=0}^{\infty} y^n  \sum_{r=0}^n \sum_{j=0}^{n-r} \frac{P_{r+1}}{(r+1)! j! (n-r-j)!}  \left[\sum_{l=0}^{\infty}  \sum_{k=0}^{\infty} \binom{k+r}{l} u^l w^{k+n-l}\right] \left[\sum_{i=j}^{r+j} \binom{r}{i-j} x^i \right] \\
        & = & - \frac{1}{1-(w+u)} \sum_{n=0}^{\infty} y^n  \sum_{r=0}^n \sum_{j=0}^{n-r} \frac{P_{r+1}}{(r+1)! j! (n-r-j)!}  w^{n-r} (w+u)^r x^j (1+x)^r,
    \end{Eqnarray*}
since $\binom{r}{i-j}$ vanishes unless $j \leq i \leq r+j$ and $r+j \leq n$. Exchanging limits twice using (\ref{eq:sum_nr}), once for $r$ and once for $j$,
    \begin{Eqnarray*}
        h_7 & = & - \frac{1}{1-(w+u)} \sum_{r=0}^\infty \sum_{n=0}^{\infty} (yw)^n   \sum_{j=0}^{n} \frac{P_{r+1}}{(r+1)! j! (n-j)!} z^r x^j\\
        & = & - \frac{1}{1-(w+u)} \left[\sum_{n=0}^{\infty} \frac{(yw)^n}{n!}\right] \left[\sum_{r=0}^\infty \frac{P_{r+1}}{(r+1)!} z^r \right] \left[\sum_{j=0}^\infty     \frac{(ywx)^j}{j!} \right] \\
        & = & - \frac{\ee^{yw(1+x)}}{1-(w+u)} \frac{\ee^{z}}{\ee^{z}+1}.
    \end{Eqnarray*}
Along similar lines we manipulate
    \begin{Eqnarray*}
        h_8
        & = & \sum_{l=0}^{\infty} u^l\sum_{k=0}^{\infty} w^k \sum_{n=0}^{k+l}  y^n   \sum_{r=0}^n \sum_{j=0}^{n-r} \frac{- P_{r+1}}{(r+1)!j!(n-r-j)!} \binom{k+l-j}{l} \left[\sum_{i=0}^{n} \binom{r}{i-j} x^i \right]\\
        & = & \sum_{n=0}^{\infty}   \sum_{r=0}^n \sum_{j=0}^{n-r} \frac{-  P_{r+1} y^n}{(r+1)!j!(n-r-j)!} \!\! \left[\sum_{i=j}^{r+j} \binom{r}{i-j} x^i \right] \!\! \left[\sum_{l=0}^{\infty}  \sum_{k=0}^{\infty} \binom{k+n-j}{l} u^l w^{k+n-l}\right]\\
        & = & - \frac{1}{1-(w+u)} \sum_{n=0}^{\infty} y^n  \sum_{r=0}^n \sum_{j=0}^{n-r} \frac{P_{r+1}}{(r+1)!j!(n-r-j)!}  x^j (1+x)^r w^j (w+u)^{n-j},
    \end{Eqnarray*}
since $\binom{r}{i-j}$ vanishes unless $j \leq i \leq r+j$ and $r+j \leq n$. Once again, exchanging limits twice using (\ref{eq:sum_nr}), once for $r$ and once for $j$,
    \begin{Eqnarray*}
        h_8
        & = & - \frac{1}{1-(w+u)} \sum_{r=0}^\infty \sum_{n=0}^{\infty} y^{n+r}   \sum_{j=0}^n \frac{P_{r+1}}{(r+1)!j!(n-j)!}  x^j (1+x)^r w^j (w+u)^{n+r-j} \\
        & = & - \frac{1}{1-(w+u)} \left[\sum_{n=0}^{\infty} \frac{(y(w+u))^n}{n!} \right] \left[ \sum_{r=0}^\infty \frac{P_{r+1}}{(r+1)!} z^r \right] \left[\sum_{j=0}^\infty \frac{(xyw)^j}{j!} \right]   \\
        & = & - \frac{\ee^{y(w+u)+xyw}}{1-(w+u)} \frac{\ee^{z}}{\ee^{z}+1}.
    \end{Eqnarray*}
These expressions are then combined to form the full generating function
    \begin{Eqnarray*}
        2 h(u,w,y,x)
        & = & \frac{\ee^{yw}}{1-(w+u)} + \frac{\ee^{y(w+u)}}{1-(w+u)} + \frac{\ee^{yw(1+x)}}{1-(w+u)}+\frac{\ee^{y(w+u)+xyw}}{1-(w+u)}\\
        &&- \frac{\ee^{yw}}{1-(w+u)} \frac{\ee^{z}}{\ee^{z}+1} - \frac{\ee^{y(w+u)}}{1-(w+u)} \frac{\ee^{z}}{\ee^{z}+1}\\
        && - \frac{\ee^{yw(1+x)}}{1-(w+u)} \frac{\ee^{z}}{\ee^{z}+1} - \frac{\ee^{y(w+u)+xyw}}{1-(w+u)} \frac{\ee^{z}}{\ee^{z}+1}\\
        & = & \frac{\ee^{yw} + \ee^{y(w+u)} + \ee^{yw(1+x)} + \ee^{y(w+u)+xyw}}{1-(w+u)}  \left[ 1 - \frac{\ee^{z}}{\ee^{z}+1} \right] \\
        & = & \frac{1}{1-(w+u)} \left[\frac{\ee^{yw} + \ee^{y(w+u)} + \ee^{yw(1+x)} + \ee^{y(w+u)+xyw}}{\ee^{z} +1}\right].
    \end{Eqnarray*}
For further simplification, consider $g=2h(u,w,2y,x)(1-(w+u))$, letting $a=uy, b = wy, c = uxy$ and $d = wxy$,
    \begin{Eqnarray*}
        g&=& \frac{\ee^{2yw} + \ee^{2y(w+u)} + \ee^{2yw(1+x)} + \ee^{2y(w+u)+2xyw}}{\ee^{2yw+2ywx+2yu+2yux} +1} \\
        &=& \frac{\ee^{2b} + \ee^{2a+2b} + \ee^{2b + 2d} + \ee^{2a+2b+2d}}{\ee^{2a+2b+2d+2c} +1} \\
        &=& \frac{\ee^{-a +b -c -d} + \ee^{a +b -c -d} + \ee^{-a +b-c +d} + \ee^{a +b -c +d} }{\ee^{a+b+c+d}+\ee^{-a-b -c -d}}\\
        &=& \ee^{b-c} \frac{\ee^{-a -d} + \ee^{a -d} + \ee^{-a +d} + \ee^{a +d} }{\ee^{a+b+c+d}+\ee^{-a-b -c -d}}\\
        &=& \ee^{b-c} \frac{(\ee^{-a} + \ee^{a}) (\ee^{-d} +\ee^{d})}{\ee^{a+b+c+d}+\ee^{-a-b -c -d}}\\
        &=& 2 \ee^{b-c} \frac{\cosh(a)\cosh(d)}{\cosh(a+b+c+d)}.
    \end{Eqnarray*}
This brings us to the desired form of the generating function,
    \begin{equation*}
        \tag{\ref{eq:genfun_sol_pi}}
        h(u,w,y,x)  = \frac{\exp\left((wy-uxy)/2\right) }{1-(w+u)} \frac{\cosh(uy/2)\cosh(wxy/2)}{\cosh(y(u+w)(1+x)/2)},
    \end{equation*}
and completes the proof of Lemma~\ref{lem:genfun_pi}.

\end{document}